\documentclass[reqno]{amsart}
\usepackage{amsfonts}
\usepackage{amsthm}
\usepackage{amstext}
\usepackage{amsmath}
\usepackage{amscd}
\usepackage{amssymb}
\usepackage[mathscr]{eucal}
\usepackage{url}
\usepackage{graphicx}
\usepackage{verbatim}
\usepackage{enumitem}
\usepackage{hyperref}
\usepackage{mathrsfs}
\usepackage[T1]{fontenc}

\newtheorem{theorem}{Theorem}[section]
\newtheorem{corollary}[theorem]{Corollary}
\newtheorem{proposition}[theorem]{Proposition}
\newtheorem{lemma}[theorem]{Lemma}

\newtheorem{maintheorem}{Theorem}

\newtheorem{maincorollary}[maintheorem]{Corollary}

\theoremstyle{definition}
\newtheorem{definition}[theorem]{Definition}
\newtheorem{example}[theorem]{Example}

\newtheorem*{remark*}{Remark}

\numberwithin{equation}{section}

\newcommand\mL{\v{L}}

\newcommand{\eps}{\varepsilon}

\newcommand{\RRR}{{\mathbb R}}
\newcommand{\ZZZ}{{\mathbb Z}}
\newcommand{\NNN}{{\mathbb N}}
\newcommand{\AAa}{{\mathcal A}}
\newcommand{\BBb}{{\mathcal B}}
\newcommand{\CCc}{{\mathcal C}}
\newcommand{\HHh}{{\mathcal H}}
\newcommand{\SSs}{{\mathscr{S}}}
\newcommand{\XXx}{{\mathcal{X}}}

\newcommand{\diam}{\mathop{\rm diam}}
\newcommand{\Bd}{\mathop{\rm Bd}}

\newcommand{\card}{\mathop{\rm card}}
\newcommand{\proj}{\operatorname{proj}}
\newcommand{\End}{\operatorname{End}}
\newcommand{\Cut}{\operatorname{Cut}}
\newcommand{\Branch}{\operatorname{B}}

\newcommand{\Fix}{\operatorname{Fix}}
\newcommand{\Per}{\operatorname{Per}}
\newcommand{\Orb}{\operatorname{Orb}}

\newcommand{\Prox}{\mathop{\rm Prox}}
\newcommand{\nAs}{\mathop{\rm nAs}}
\newcommand{\LY}{\mathop{\rm LY}}
\newcommand{\nLY}{\mathop{\rm nLY}}

\newcommand{\abs}[1]{\lvert#1\rvert}
\newcommand{\ord}{\operatorname{ord}}

\begin{document}
\begin{large}

\title[Generic chaos on dendrites]{Generic chaos on dendrites}

\author{\mL ubom\'\i r Snoha}
\author{Vladim\'\i r \v Spitalsk\'y}
\author{Michal Tak\'acs}
\address{Department of Mathematics, Faculty of Natural Sciences, Matej Bel University,
Tajovsk\'{e}ho 40, 974 01 Bansk\'{a} Bystrica, Slovakia}

\email{lubomir.snoha@umb.sk, vladimir.spitalsky@umb.sk, 
michal.takacs.math@gmail.com}

\thanks{This work was supported by the Slovak Research and Development Agency 
under contract No.~APVV-15-0439 and by VEGA grant 1/0158/20.}



\dedicatory{Dedicated to the memory of Sylvie Ruette}

\begin{abstract} 
We characterize dendrites $D$ such that a continuous selfmap of $D$ is 
generically chaotic (in the sense of Lasota) if and only if it is generically $\eps$-chaotic for some 
$\eps>0$. In other words, we characterize dendrites on which
generic chaos of a continuous map can be described in terms of the behaviour of
subdendrites with nonempty interiors under iterates of the map.
A dendrite $D$ belongs to this class if and only if it is 
completely regular, with all points of finite order
(that is, if and only if $D$ contains neither a copy of the Riemann dendrite nor a copy of the $\omega$-star).
\end{abstract}

\maketitle



\section{Introduction and main results}\label{S:intro}

During the last decades many interesting connections between dynamical systems and continuum theory have been studied. To illustrate this, we mention a few results.

Handel \cite{Ha82} has constructed a $C^\infty$ area preserving diffeomorphism 
of the plane with the pseudocircle as a minimal set.

Many authors have been investigating the problem whether various classes of curves
admit positive entropy homeomorphisms. One of the first results was
that Knaster continua \cite{Barge87} and the pseudoarc \cite{Kennedy91} have this property.  
On the other hand, every homeomorphism of a regular continuum has zero entropy \cite{Seidler90}, and even all group actions on regular continua are null \cite{GM19}. By
\cite{Mo11}, also every homeomorphism of a chainable hereditarily decomposable continuum has zero entropy.

The question whether a given class of curves admits an expansive homeomorphism has also attracted significant interest.
The dyadic solenoid \cite{W55} and Plykin's attractors \cite{Ply84}
are examples of continua that do admit expansive homeomorphisms,
while tree-like continua \cite{Mo02} and hereditarily indecomposable continua \cite{KM08} do not. 
For related results concerning continuum-wise expansive homeomorphisms see \cite{Ka19} and references therein.

The topology of curves appearing as inverse limit spaces of interval maps is a very active area of research; as a general reference see \cite{IM12}. 
As an example of a deep result let us mention 
the proof of Ingram's conjecture stating that inverse limit spaces of tent maps with different slopes are nonhomeomorphic \cite{BBS12}.

One of the important classes of one-dimensional continua in dynamics are dendrites.
Dendrites have long been studied in topology \cite{Why1942,Kur68}
and it is of interest that they appear also in complex dynamics as Julia sets (see e.g.~\cite{LS98}).
For us it is important that new dynamical phenomena,  which are not possible on graphs, appear on dendrites.
For instance, Wa\.{z}ewski's universal dendrite
admits a weakly mixing, nonmixing system \cite{HM14} which is proximal \cite[Theorem~5.2]{AHNO17} 
and has zero entropy \cite{BFK17}; such an example is not possible
on dendrites having nondense branch points \cite{DSS2013}. 
The topology of dendrites admits $\omega$-limit sets
and minimal sets that are more complex than in simpler spaces;
see \cite{Spi08} and \cite{BDHSS09}, respectively, 
for their topological characterizations.

Since there are many dynamical properties satisfied by all tree maps but not by all dendrite maps, it is natural to try to characterize the class of dendrites $X$ such that every dynamical system on $X$ possesses a given property. 
For example, in \cite{Il98} it is proved that a dendrite
has the $PR$-property (i.e., the closure of the set of periodic points equals the closure of the set of recurrent points for every continuous selfmap) if and only if
it has (at most) countably many endpoints.
By \cite{Ma07},
a dendrite has the $\Omega EP$-property (i.e., the set of nonwandering points
is contained in the closure of the set of eventually periodic points for every continuous selfmap) if and only if it does not contain a null comb.
(A \emph{null comb} is any dendrite homeomorphic to the subgraph of
the map $f\colon [0,1]\to [0,1]$ given by $f(x)=x$ for $x=1/n$, $n\in\NNN$, and
$f(x)=0$ otherwise; 
the dendrite in Figure~\ref{fig:comb} 
is the union of a null comb and two arcs.)

The present paper provides another link between point set topology (in particular continuum theory) and topological dynamics (in particular topological chaos). Namely,
we characterize dendrites on which generic chaos is equivalent with generic $\eps$-chaos, see Theorem~\ref{T:dendrite_gch=gech} and the comment below it.

The notion of chaos in~connection with a map was first introduced by Li and 
Yorke in~\cite{LY75}, although they did not give a formal definition. As 
of today, Li-Yorke chaos is understood in~the following way: for a 
dynamical system $(X,f)$, $X$ being a compact metric space with metric $d$ and 
$f$ being a continuous map $X\rightarrow X$, a pair $(x,y)$ of points 
in~$X$ is called a \emph{scrambled pair} if $\liminf_{n\rightarrow 
	\infty}d(f^n(x),f^n(y))=0$ and $\limsup_{n\rightarrow 
	\infty}d(f^n(x),f^n(y))>0$. For every scrambled pair $(x,y)$ there is an $\varepsilon >0$ such that $\limsup_{n\rightarrow 
		\infty}d(f^n(x),f^n(y))> \varepsilon$ and then it is called an $\varepsilon$-scrambled pair. A set $S\subseteq X$ is \emph{scrambled} or $\varepsilon$-scrambled if every 
	pair $(x,y)$ of distinct points in~$S$ is scrambled or $\varepsilon$-scrambled, respectively. The system $(X,f)$ is \emph{Li-Yorke chaotic} or \emph{Li-Yorke $\varepsilon$-chaotic} if it has an uncountable scrambled  or \emph{$\varepsilon$-scrambled} set, respectively.

Li-Yorke chaos fits particularly well into interval dynamics, since maps of 
type $2^n$ in~the Shar\-kov\-sky ordering are not Li-Yorke chaotic, maps of type 
greater than $2^{\infty}$ are Li-Yorke chaotic, and maps of type $2^{\infty}$ 
may or may not be Li-Yorke chaotic. Moreover, the existence of 
just one scrambled pair on the interval implies the existence of a Cantor $\varepsilon$-scrambled set for some $\varepsilon >0$, and hence Li-Yorke  $\varepsilon$-chaos~\cite{KS1989}. Also, as shown in~\cite{Smi1986}, Li-Yorke chaos 
turns out to be the minimal requirement for a continuous selfmap of an interval 
to be ``chaotic"; indeed, an interval map $f$ is either Li-Yorke $\varepsilon$-chaotic for some $\varepsilon >0$, or all 
trajectories of $f$ are approximable by periodic orbits. All of this shows that, on 
the interval, Li-Yorke chaos is quite a natural notion, albeit a very weak form of chaos.  

Outside the interval, Li-Yorke chaos seems to be less natural. For instance, though 
the existence of a scrambled pair implies the existence of an uncountable (in this case even a Cantor) $\varepsilon$-scrambled set also on graphs~\cite{RS2014}, this is no longer true already for dendrite maps \cite{Koc2012} and for triangular maps in the square~\cite{FPS1999}. Though on graphs Li-Yorke chaos and Li-Yorke $\varepsilon$-chaos coincide~\cite{RS2014}, this is not true in general, as for instance Floyd's minimal system shows~\cite{Fl49},\cite[pp. 24--27]{Aus88}; this is a minimal homeomorphism on a nonhomogeneous space whose connected components are singletons and arcs. The maximal scrambled sets coincide with those arcs (cf.~\cite[Proposition~2]{CSS04}) but there is no infinite $\varepsilon$-scrambled set for any $\varepsilon>0$ due to linearity of the map on the arcs. 
Using techniques from \cite[Subsection~2.3]{CSS04} 
one can embed Floyd's system into the Cantor fan (i.e., the cone over the Cantor set) in such a way that the new system still has no 
infinite $\varepsilon$-scrambled set. 
As regards dendrites, in the appendix we show that even every dendrite with uncountably many endpoints
admits a Li-Yorke chaotic map with no infinite $\eps$-scrambled set, see Proposition~\ref{P:LY-chaos}.

One can ask why uncountability, and not 
some other `size', is required in the definition of Li-Yorke chaos. A partial clarification for this can be found in~\cite{BHS08}, where it is shown that in many 
cases (not only for interval and graph maps) the uncountability of a scrambled set implies the existence of a Cantor scrambled set, which is perhaps a more natural choice of a `large' scrambled set. Whether the existence 
of an uncountable scrambled set implies the existence of a Cantor scrambled set 
in general remains an open problem \cite{BHS08}. However, the existence of an uncountable $\varepsilon$-scrambled set in a Polish space implies the existence of a Cantor $\varepsilon$-scrambled set~\cite[Theorem~16]{BHS08}.

It seems that Lasota was one of the first searching for a stronger 
type of formally defined chaos. The notion of generic chaos, suggested by him (as claimed by his 
student Pi\'orek in~\cite{Pio85}), is defined by the requirement for the set 
of all scrambled pairs to be generic, i.e., residual, in the square $X\times X$ 
(recall that a subset of a metric space is residual if its complement is a set 
of the first category, i.e., the union of countably many nowhere dense sets).

It will be convenient to use the following notation. For $\varepsilon >0$,
\begin{eqnarray*}
	\Prox(f)&=&\{(x,y)\in X^2:\liminf_{n\rightarrow 
		\infty}d(f^n(x),f^n(y))=0\}, \\
	\nAs(f)&=&\{(x,y)\in X^2:\limsup_{n\rightarrow \infty}d(f^n(x),f^n(y))>0\}, 
	\\
	\nAs(f,\eps)&=&\{(x,y)\in X^2:\limsup_{n\rightarrow 
		\infty}d(f^n(x),f^n(y))>\eps\},\\
	\LY(f)&=&\Prox(f)\cap \nAs(f), \\
	\LY(f,\eps)&=&\Prox(f)\cap \nAs(f,\eps).
\end{eqnarray*}
Thus, $\Prox(f)$ is the \emph{proximal relation}, $\nAs(f)$ is the complement of the 
\emph{asymptotic relation} and $\LY(f)$ is sometimes called the \emph{Li-Yorke relation} in the 
considered dynamical system. The elements of $\LY(f)$ are Li-Yorke or scrambled 
pairs, and the elements of $\LY(f,\eps)$ are $\varepsilon$-Li-Yorke or 
$\varepsilon$-scrambled pairs. Thus, Lasota's definition is as follows. 
\begin{itemize}
	\item A map $f:X\rightarrow X$ is called \emph{generically chaotic} if the set 
	$\LY(f)$ is residual in $X^2$. 
\end{itemize}

One of the closely related notions introduced and studied in~\cite{Sno1990, 
	Sno1992} is the following (recall also that Ruette 
simplified some ideas and results from~\cite{Sno1990, Sno1992} and answered some 
open questions posed there, see e.g.~\cite{Rue13}). 
\begin{itemize}
	\item A map $f:X\rightarrow X$ is called \emph{generically $\eps$-chaotic} if 
	the set $\LY(f,\eps)$ is residual in $X^2$.
\end{itemize}

The first examples of generically chaotic interval maps were found by Pi\'orek~\cite{Pio85}.
Many examples follow from the fact, contained implicitly in \cite[Theorem 
3.5]{HY02} and explicitly in \cite[Proposition 2.11]{BGKM2002}, that a weakly 
mixing map on a nondegenerate compact metric space $X$ is generically 
$\eps$-chaotic for every $0<\eps<\diam X$. Examples of generically 
$\eps$-chaotic maps that are not weakly mixing can be found in~\cite{Sno1990}.

The way generic chaos is defined can be viewed as a `microscopic' definition in 
the following sense. In order to verify by the definition whether a system is 
generically chaotic, we should investigate trajectories of pairs of `tiny' 
points in the whole space. However, it appears to be an almost impossible task to 
determine whether residually many of them are Li-Yorke pairs. 
Snoha~\cite{Sno1990} found a way to verify, on 
the interval, generic chaoticity in `macroscopic' terms. More precisely, he 
showed that on the interval the notions of generic chaos and generic  
$\varepsilon$-chaos are equivalent, and that the latter (and hence, also the former) can 
be checked by investigating the dynamics at a `macroscopic' level, namely by 
investigating the trajectories of the (nondegenerate) subintervals rather than 
points. To study the behaviour of all subintervals and their pairs under  
iterates of the map is of course much easier than to do the same with residually 
many pairs of
points. Therefore, it is not surprising that, as demonstrated in~\cite{Sno1990}, 
this approach often enables successful verification of `microscopically' defined 
generic chaoticity of an interval map purely `macroscopically'. The situation 
seems to be similar to that of the point transitivity, which is also 
defined `microscopically' (as the existence of a point with dense orbit) yet, 
in nice spaces, is equivalent with topological transitivity and so can be 
verified `macroscopically' by examining the behaviour of nonempty open sets 
under the iterates of the map.

In \cite{Sno1992}, Snoha also remarked that 
it is possible to carry over some results concerning generic 
chaos from the interval to compact metric spaces. Then, 
Murinov\'a~\cite{Mur00} proved these generalized statements for a class of 
metric spaces containing, in particular, all compact metric spaces; to state the 
main result of \cite{Mur00} (see Proposition~\ref{P:B1-B2} below), we first introduce some 
conditions. 

If $(X,f)$ is a dynamical system on a compact metric space $X$ with metric $d$, 
and if $\SSs$ is a family of nondegenerate subsets of $X$, we will consider
the following three conditions which the family $\SSs$ may or may not satisfy:

\begin{enumerate}[label=({Prox}), left=2\parindent]
	\item\label{ENUM:B1} 
	for all sets $S_1,S_2\in\SSs$,
	$$
	\liminf_{n\to\infty} d(f^n(S_1),f^n(S_2))=0;
	$$ 
\end{enumerate}
\begin{enumerate}[label=({Sens}$_0$), left=2\parindent]
	\item\label{ENUM:Sens0} 
	for every set $S\in\SSs$,
	$$
	\limsup_{n\rightarrow \infty}\diam f^n(S)>0;
	$$   
\end{enumerate}
\begin{enumerate}[label=({Sens}), left=2\parindent]
	\item\label{ENUM:B2} 
	there is $\eta>0$ such that, 
	for every set $S\in\SSs$,
	\begin{equation}\label{EQ:Sens-eta}
	\limsup_{n\rightarrow \infty}\diam f^n(S)>\eta.
	\end{equation}
\end{enumerate}
Thus, \ref{ENUM:B1} means that all sets $S_1,S_2\in\SSs$ are proximal.
The other two conditions are close to sensitivity. If $\SSs$ is the family of
all open balls, then \ref{ENUM:B2} is in fact equivalent with the sensitivity of the system.\footnote{A
	system $(X,f)$ is \emph{sensitive} if there is $\eps>0$ such that, for every $x\in X$
	and every $\delta>0$, there is $y\in X$ with $d(y,x)<\delta$ such that
	$d(f^n(y), f^n(x))\ge\eps$ for some $n\ge 0$ (then, by choosing $y$ in a smaller neighbourhood of $x$ 
	if necessary, we may assume that $n$ is as large as we require).}
For this family $\SSs$, pointwise Lyapunov instability implies \ref{ENUM:Sens0},
but the converse is not true (for instance, the identity satisfies \ref{ENUM:Sens0}).\footnote{A system 
	$(X,f)$ is \emph{pointwise Lyapunov unstable} if for every $x\in X$ there exists $\eps>0$
	such that for every $\delta>0$ there is $y\in X$ with $d(y,x)<\delta$ such that
	$d(f^n(y), f^n(x))\ge\eps$ for some $n\ge 0$.} 

The above three conditions are closely related to generic chaos. First, realize that 
trivially
\begin{equation*}
\boxed{
	\begin{split}
	\text{$f$ is generically } & \text{chaotic} \\
	& \Longrightarrow \text{\ref{ENUM:B1} and 
		\ref{ENUM:Sens0} \, hold for the family of all open balls,}
	\end{split}
}
\end{equation*}

\smallskip

\noindent because generic chaoticity of $f$ implies that both $\Prox(f)$ and 
$\nAs(f)$ are dense in $X^2$. The converse implication 
does not hold even on the interval, see~\cite[Example 3.6]{Sno1990}. However, 
we have the equivalence\footnote{Of course, in the two boxed statements, 
	the family of all open balls may be replaced by the 
	family of all closed balls or by the family of all nonempty open sets.}

\begin{equation*}
\boxed{
	\begin{split}
	\text{$f$ is generically } & \text{$\eps$-chaotic for some 
		$\eps >0$} \\
	& \Longleftrightarrow  
	\text{\ref{ENUM:B1} and \ref{ENUM:B2} \, hold for the family of all open 
		balls.}
	\end{split}
}
\end{equation*}

\smallskip

\noindent More precisely, Theorem~A from \cite{Mur00} implies the following.

\begin{proposition}[\cite{Mur00}]  \label{P:B1-B2}
	Let $X=(X,d)$ be a compact metric space and $f\colon X\to X$ be continuous. 
	Then the following are equivalent:
	\begin{enumerate}
		\item $f$ is generically $\eps$-chaotic for some $\eps>0$;
		\item \ref{ENUM:B1} and \ref{ENUM:B2} are satisfied by the family 
		of all open (or closed) balls in $X$.
	\end{enumerate}
	Moreover, if $f$ is generically $\eps$-chaotic then \ref{ENUM:B2} holds, 
	for the family of all open (or closed) balls in $X$, with $\eta=\eps$. 
	Conversely, if \ref{ENUM:B1} and \ref{ENUM:B2} are satisfied by the family 
	of all open (or closed) balls in $X$, then $f$ is generically 
	$\eps$-chaotic for any $\eps < \eta/2$.
\end{proposition}

The two boxed statements are crucial for understanding some proofs in the present paper.
We will often use them without explicitly citing them. Also, we hope that no misunderstanding
will arise if, for a generically chaotic map, we say that something is true by, for instance, 
\ref{ENUM:Sens0}. By this we of course mean that, due to the first boxed statement, 
for a generically chaotic map,
any set with nonempty interior satisfies the inequality from \ref{ENUM:Sens0}.

\smallskip

Thus, in spaces where generic chaos is equivalent to generic $\varepsilon$-chaos, 
we can check generic chaos `macroscopically' using \ref{ENUM:B1} and 
\ref{ENUM:B2} for open or closed balls. We know from~\cite{Sno1990} that 
the interval is such a space.
By~\cite{Tak2016}, even graphs are such spaces. 
However, in general, this is not true for dendrites; indeed, 
by \cite{Mur00}, an $\omega$-star admits a generically 
chaotic selfmap that is not generically $\varepsilon$-chaotic for any 
$\varepsilon >0$ (an \emph{$\omega$-star} is a (topologically unique) dendrite 
having exactly one branch point, and this branch point is of infinite order).

The two boxed statements show that the difference between generic chaos and generic $\eps$-chaos lies in the sensitivity. Indeed, $f$ is generically $\eps$-chaotic for some $\eps>0$ if and only if it is generically chaotic and sensitive (i.e., satisfies \ref{ENUM:B2}). Thus, the above mentioned map on an $\omega$-star shows that a generically chaotic map need not be sensitive. Another such example was suggested by an anonymous referee, see Example~\ref{Ex:gch-not-sensitive}.

Thus, already on dendrites, generic $\varepsilon$-chaos is stronger than generic 
chaos.\footnote{This is not surprising; a similar fact is that for a continuous 
	selfmap of a compact metric space, the whole space can be a 
	scrambled set~\cite{HY01}, while it is never an $\varepsilon$-scrambled 
	set~\cite[Proposition 5]{BHS08} for any $\varepsilon>0$ (compactness is 
	essential here, see~\cite[Example 6]{BHS08}).} 
The main aim of the present 
paper is to \emph{characterize} dendrites on which generic chaos is equivalent 
to generic $\varepsilon$-chaos for some $\eps>0$, i.e., to characterize 
dendrites on which generic chaos can be verified `macroscopically', using the 
conditions \ref{ENUM:B1} and \ref{ENUM:B2}.

We know that on an $\omega$-star there exists a 
generically chaotic map which is not generically $\eps$-chaotic for any 
$\eps>0$. By generalizing the construction from \cite{Mur00}, we 
will show that such a map can be constructed 
on every dendrite having a branch point of infinite order, see Lemma~36. Thus, for all generically chaotic maps to be generically $\eps$-chaotic, every branch point of a dendrite must be of finite order.

In Lemma~36, we will further show 
that another condition necessary for obtaining the equivalence between 
generic chaos and 
generic $\eps$-chaos on a dendrite is complete regularity.
Recall that a continuum is \emph{completely regular} if every nondegenerate 
subcontinuum of it has nonempty interior (\cite{Nik89}, see also 
\cite{Ili80}). A singleton is completely regular for trivial reasons. A subcontinuum of 
a completely regular continuum is completely regular.
By \cite[Theorem~4 on p.~301, Theorem~3 on p.~284]{Kur68}, every dendrite as 
well as every completely regular continuum is regular (a continuum is 
\emph{regular} if it has a basis consisting of open sets with finite boundary). 
Wa\.{z}ewski's universal dendrite is an example of a dendrite which is not 
completely regular.
It is interesting that by \cite{OmZaf05} there is a universal completely 
regular dendrite, that is, a completely regular dendrite containing a copy of 
every completely regular dendrite.

Our main result is the following theorem showing that the two necessary 
conditions, when taken together, are also sufficient. To prove that generic chaos together with these two conditions imply generic $\eps$-chaos, we first study invariant subcontinua for a generically chaotic map on a dendrite satisfying the two conditions. Similarly as on the interval \cite{Sno1990} and graphs \cite{Tak2016}, we are able to prove that the invariant nondegenerate subdendrites have large diameters, see Lemma~\ref{L:subdend_at_least_delta}. Of course, the topology of dendrites makes the proof much more complicated, among other reasons due to the phenomenon described in Footnote~\ref{Footnote-dend}.
Contrary to the interval and graph case, there is still a long road to finish the proof. One of the new ideas is that now we need to use a nontrivial result from \cite{Blo2013} about a dichotomy for the character of fixed points of continuous dendrite maps, see Lemma~\ref{L:noneffluent_endpoint}.
This result is crucial in Lemmas~\ref{L:orbits_with_fD_intersects_D-diameter_delta}
and \ref{L:k_equals_1}, in which we study subdendrites
having orbits containing no fixed and periodic points, respectively (the interval and graph case did not require anything like this). Even using these two lemmas, the proof is still quite long.

\begin{maintheorem}\label{T:dendrite_gch=gech}
	Let $D$ be a dendrite. Then the following are equivalent.
	\begin{enumerate}
		\item \label{ENUM:T1} 
		For every continuous map $f\colon D\to D$, $f$ is generically chaotic if and only if 
		it is generically $\eps$-chaotic for some $\eps>0$.
		\item \label{ENUM:T2} 
		The dendrite $D$ is completely regular, with all points of finite order.
	\end{enumerate}
\end{maintheorem}

Condition \eqref{ENUM:T2} can be reformulated by saying
that the dendrite $D$ contains neither a copy of the Riemann dendrite (see Proposition~\ref{P:conditions} and the definition of the Riemann dendrite just above it) nor a copy of the $\omega$-star.
Notice that if $D$ is a singleton, then both \eqref{ENUM:T1} and \eqref{ENUM:T2}  trivially hold.

Recall that if $X$ is a compact metric space and $A\subseteq X$ is an arc (i.e., a homeomorphic image of the interval $[0,1]$), then $A$ is called a \emph{free arc} if the arc $A$ without its endpoints is an open set in $X$, see e.g.~\cite{DSS2013}.
Proposition~\ref{P:B1-B2} immediately gives the following corollary of Theorem~\ref{T:dendrite_gch=gech}
(see also Proposition~14).

\begin{maincorollary}\label{C:main}
	Let $D$ be a nondegenerate completely regular dendrite with all points of finite order. Let $f\colon D\to D$ be a continuous map. Then the following conditions are equivalent:
	\begin{enumerate}
		\item\label{ENUM:Cor1} $f$ is generically chaotic;
		\item\label{ENUM:Cor2} $f$ is generically $\eps$-chaotic for some $\eps>0$;
		\item\label{ENUM:Cor3} \ref{ENUM:B1} and \ref{ENUM:B2} are satisfied by the family of all open (or closed) 
		balls in $D$;
		\item\label{ENUM:Cor4} \ref{ENUM:B1} and \ref{ENUM:B2} are satisfied by the family of all free arcs in $D$;
		\item\label{ENUM:Cor5} \ref{ENUM:B1} and \ref{ENUM:B2} are satisfied by the family of all nondegenerate subdendrites of $D$.
	\end{enumerate}
	Moreover, if $f$ is generically $\eps$-chaotic then \ref{ENUM:B2} in \eqref{ENUM:Cor3}--\eqref{ENUM:Cor5} hold with $\eta=\eps$. Conversely, if \ref{ENUM:B1} and \ref{ENUM:B2} are satisfied in one of \eqref{ENUM:Cor3}--\eqref{ENUM:Cor5}, then $f$ is generically $\eps$-chaotic for any $\eps<\eta/2$.
\end{maincorollary}

Using Proposition~\ref{P:B1-B2}, one can observe that conditions \eqref{ENUM:Cor2}--\eqref{ENUM:Cor4} in Corollary~\ref{C:main} are equivalent for a larger class of dendrites, namely for dendrites with dense free arcs
(an example of such a dendrite is the Riemann dendrite, see Figure~1).

\medskip
In connection with Theorem~\ref{T:dendrite_gch=gech}, 
a natural question is whether the equivalence between generic chaos and generic $\eps$-chaos
can be extended to a larger class of spaces.
First of all, realize that some compact metric spaces $X$ do not admit generically chaotic maps at all.
For instance, this is true if $X$ is \emph{rigid} (i.e., $X$ admits 
no continuous selfmap other than the identity or a constant map). 
This is also true if $X$ has an isolated point 
(an isolated point is a nonempty open set which does not satisfy the inequality from \ref{ENUM:Sens0}).

We denote by $\XXx$ the system of all compact metric spaces $X$ admitting at least one generically chaotic map
and satisfying the condition
\[
f\colon X\to X \text{ is generically chaotic } \iff f \text{ is generically }\varepsilon\text{-chaotic for some } \varepsilon >0.
\]
Thus, $\XXx$ is the system of those compacta on which generic chaos of continuous selfmaps 
can be checked macroscopically in the sense discussed above, i.e., using \ref{ENUM:B1} and \ref{ENUM:B2}.

The system $\XXx$ contains no zero-dimensional compact metric space 
(we already know this if it contains an isolated point;
otherwise, see Proposition~\ref{P:Cantor}).
We now consider one-dimensional spaces.
By \cite{Tak2016}, $\XXx$ contains all finite graphs and, by our Theorem~\ref{T:dendrite_gch=gech},
it contains many, but not all, dendrites.
We conjecture that Theorem~\ref{T:dendrite_gch=gech} can be extended to local dendrites,
i.e., we conjecture that a local dendrite is in $\XXx$ if and only if it is completely regular, with all points of finite order. More generally, one can ask which locally connected curves belong to $\XXx$.

The square is an example of a two-dimensional space which admits a generically chaotic map but is not in $\XXx$,
see the end of the proof of Theorem~B in \cite{Mur00}.

\medskip
The paper is organized as follows.
Sections~2 and 3 are preliminary.
Section~4 contains technical lemmas which are then used
in Section~5 to prove the implication (\ref{ENUM:T2}) $\Rightarrow$ (\ref{ENUM:T1})
in Theorem~\ref{T:dendrite_gch=gech}.
In Section~6, we construct special exact maps on dendrites which are
not completely regular.
Such maps are used in Section~7 to prove the implication (\ref{ENUM:T1}) $\Rightarrow$ (\ref{ENUM:T2}).

\section{Preliminaries}\label{S:prelim}

Let $\NNN$ be the set of all positive integers and $\NNN_0=\{0\}\cup\NNN$.
Let $I$ be the unit interval $[0,1]$.
We sometimes use the symbol $\sqcup$ to denote a disjoint union.
The cardinality of a set $A$ is denoted by $\card{A}$.
By $\Bd A$ we denote the boundary of $A$.

If $X=(X,d)$ is a metric space and $A,B$ are subsets of it, by $d(A,B)$ we mean the distance from $A$ to $B$,
that is, $d(A,B)=\inf\{d(a,b)\colon a\in A,\, b\in B\}$; if $A=\{a\}$ is a singleton we write
$d(a,B)$ instead of $d(\{a\},B)$. The open $\eps$-ball centered at $x$ is denoted by $B_\eps(x)$.
A \emph{regular closed set} is a set which is equal to the closure of its interior.
A family of subsets of $X$ is called a \emph{null family} if for every $\eps>0$ 
only finitely many sets from this family have diameters larger than $\eps$.

Let $(X,f)$ be a \emph{dynamical system}, that is, $X$ is a compact metric space and $f\colon X\to X$ 
is a continuous map. 
The sets of all fixed points and all periodic points
are denoted by $\Fix(f)$ and $\Per(f)$, respectively.
The \emph{orbit} of a set $A\subseteq X$ is $\Orb_f(A)=\bigcup_{n=0}^\infty f^n(A)$.
We say that $f$ is 
\begin{itemize}
	\item \emph{transitive} if for every nonempty open sets $U,V\subseteq X$
	there is $n\in\NNN$ with $f^n(U)\cap V\ne\emptyset$ (then there are infinitely many such positive integers $n$);
	\item \emph{totally transitive} if $f^n$ is transitive for every $n\in\NNN$;
	\item \emph{weakly mixing} if $f\times f\colon X\times X\to X\times X$ is transitive;
	\item \emph{strongly mixing} if for every nonempty open sets $U,V\subseteq X$
	there is $n_0\in\NNN$ such that $f^n(U)\cap V\ne\emptyset$ for every $n\ge n_0$;
	\item \emph{exact}, or \emph{locally eventually onto}, if for every nonempty open set $U\subseteq X$
	there is $n\in\NNN$ with $f^n(U)=X$.
\end{itemize}

If $A,B$ are nonempty open subsets of $X$ and $X=A\sqcup B$, we speak on a \emph{disconnection of $X$}.
Recall that $X$ is connected if and only if it has no disconnection.
A \emph{continuum} is a nonempty, compact, connected metric space.  We say 
that a continuum is \emph{nondegenerate} if it has more than one point. A 
continuum $X$ is \emph{uniquely arcwise connected} provided that for every 
two distinct points of $X$ there is exactly one arc in $X$ joining these 
points. A continuum $X$ is \emph{unicoherent} if the intersection of every 
two of its subcontinua whose union is $X$ is connected (hence, a 
subcontinuum). $X$ is \emph{hereditarily unicoherent} if all its 
subcontinua are unicoherent. Equivalently, a continuum $X$ is hereditarily 
unicoherent if and only if the intersection of any two subcontinua of $X$ 
is connected. 

A \emph{simple closed curve} is a homeomorphic image of the unit circle. 
A \emph{dendrite} is a locally connected continuum which 
contains no simple closed curve. Note that a singleton is also a dendrite 
but the empty set is not. We will use basic facts on dendrites 
from~\cite{Kur68,Nad92,Cha98}. In particular, every 
subcontinuum of a 
dendrite is a dendrite and all dendrites are uniquely arcwise connected and 
hereditarily unicoherent.
Since a dendrite is locally connected, the components of its open subsets are open.
Recall also that dendrites have the fixed point property.

Let $D$ be a dendrite and $x\in D$. The \emph{order} of the point $x$,  
denoted by $\ord(x,D)$,  is the cardinality of the set of (connected) components of $D\setminus\{x\}$
(see \cite[Theorem~10.13]{Nad92}). 
The order can be either finite or infinite countable. 
An \emph{endpoint} is a point of order $1$. Any point with order greater than $1$
or greater than $2$ is called a \emph{cutpoint} or a \emph{branch point}, respectively. 
The sets of all endpoints, cutpoints, branch points of $D$ 
are denoted by $\End(D)$, $\Cut(D)$, $\Branch(D)$, respectively.
Notice that any nondegenerate dendrite $D$ has at least two endpoints,
and every point of it is either an endpoint or a cutpoint; 
the set $B(D)$ is countable \cite[Theorem~7, p.~302]{Kur68}
and the set $\End(D)$ is totally disconnected \cite[Theorem~2, p.~292]{Kur68}.
The unique point in a degenerate dendrite has order $0$.
For distinct $x,y\in D$ let $[x,y]$ denote the unique arc with endpoints $x$ and $y$, and let $[x,x]$ denote just the singleton $\{x\}$. For $x\neq y$ put $(x,y)=[x,y]\setminus \{x,y\}$, $(x,y]=[x,y]\setminus \{x\}$ and $[x,y)=[x,y]\setminus \{y\}$.

\begin{lemma}\label{L:endS}
Let $S$ be a nondegenerate connected set in a dendrite $D$ (hence the closure $\overline S$ is a nondegenerate dendrite). 
Let $x\in \overline{S}\setminus S$.
Then $x$ is an endpoint of $\overline{S}$.
\end{lemma}
\begin{proof}
If $x$ is a cutpoint of $\overline{S}$ then there is a disconnection $\overline{S}\setminus\{x\} = A\sqcup B$. 
Hence, since $x\notin S$, $S=(A\cap S)\sqcup (B\cap S)$ is a disconnection of $S$.
\end{proof}

\begin{lemma}\label{L:boundaryBumping}
	Let $X$ be a topological space and $U$ be an open set.
	Let $A$ be a connected subset of $X$ such that $A\cap U\ne\emptyset$ and $A\cap \Bd U=\emptyset$.
	Then $A\subseteq U$.
\end{lemma}
\begin{proof}
	Otherwise $A=(A\setminus \overline{U}) \sqcup (A\cap U)$ is a disconnection of the connected space $A$.
\end{proof}

The following simple lemma will be applied to dendrites.
\begin{lemma}\label{L:finiteCap}
	Let $E$ be a hereditarily unicoherent continuum and $E_1,\dots,E_k$ be subcontinua of $E$. 
	If the sets $E_i$ intersect pairwise, then the intersection $\bigcap_{i=1}^k E_i$ 
	is nonempty, hence it is a subcontinuum of $E$.
\end{lemma}
\begin{proof}
    We proceed by induction on $k$. For $k=1$ the lemma is trivial, for 
	$k=2$ it follows from hereditary unicoherence of $E$. Now assume that 
	$k\ge 3$ and that the lemma is valid for every family of less than $k$ 
	subcontinua of $E$. Take pairwise intersecting subcontinua 
	$E_1,\dots,E_k$. It is sufficient to prove that $\bigcap_{i=1}^k 
	E_i \neq \emptyset$ (then, due to hereditary unicoherence of $E$, this 
	intersection is a subcontinuum). By the induction hypothesis, 
	the intersections $E_1^{i}=E_1\cap\dots \cap E_{i}\ne\emptyset$ ($i< k$) are subcontinua.
	Further, since $E_{k-1}\cap E_k\ne\emptyset$, also the set $E_{k-1}\cup E_k$
	is a subcontinuum.
	Then 
	\[
	  C = E_1^{k-2} \cap (E_{k-1}\cup E_k)
	  \supseteq 
	  E_1^{k-1},
	\]
	being the nonempty intersection of two subcontinua, is a subcontinuum due to
	hereditary unicoherence.
	Now suppose that $\bigcap_{i=1}^k E_i=\emptyset$. Then, since
	$E_1^{k-2} \cap E_{k}\ne\emptyset$ by the induction hypothesis,
	\[
	  C = E_1^{k-1} \sqcup (E_1^{k-2} \cap E_{k})
	\]
	is a disconnection of $C$, a contradiction.
\end{proof}

\begin{lemma}\label{L:finiteCap2}
	Let $D$ be a dendrite and $C_1,\dots,C_k$ ($k\ge 2$) be connected subsets of $D$. 
	If $C_i\cap C_j=\emptyset\ne \overline{C_i}\cap \overline{C_j}$ for every distinct $i,j$,
	then $\bigcap_{i=1}^k \overline{C_i}$ is a singleton.
\end{lemma}
\begin{proof}
By Lemma~\ref{L:finiteCap}, $C=\bigcap_{i=1}^{k} \overline{C_i}$ is a subdendrite of $D$.
If $C$ is nondegenerate, it contains an arc $[a,b]$. Hence, due to connectedness, every $C_i$ contains $(a,b)$,
a contradiction with pairwise disjointness of them. Thus $C$ is a singleton.
\end{proof}

\begin{lemma}\label{L:arcs-and-one-point-boundary}
	Let $D$ be a dendrite and $U$ be an open connected subset with singleton boundary $\{u\}$. Let $a\in U$.
	Then, for every $x\in U$, 
	$[x,u)\cap [a,u)\ne\emptyset$.
\end{lemma}
\begin{proof}
	Suppose, on the contrary, that $[x,u)\cap [a,u)=\emptyset$ for some $x\in U$. 	
	Thus $[a,u]\cup [u,x]$ is an arc, containing $u$,
	with endpoints $a$ and $x$, hence is equal to $[a,x]$ (because $D$ is
	uniquely arcwise connected).
	However, connected sets in dendrites are arcwise connected, therefore $[a,x]\subseteq U$ and so $u\notin[a,x]$.
	This is a contradiction.
\end{proof}

\begin{lemma}\label{L:doplnkySubdendr-Hranica}
	Let $D$ be a dendrite and $E$ be a proper subdendrite of $D$.
	Then the components of $D\setminus E$ are open and 
	form a (finite or infinite countable) null family.
	Moreover, if $C$ is a component 
	of $D\setminus E$ (hence $\overline{C}$ is a subdendrite of $D$),
	then $\Bd(C)=\{c\}$ for some $c\in \Bd(E)\cap \End(\overline{C})$;
	if $E$ is nondegenerate
	then $c\in \Cut(D)$ and also $\Bd(\overline{C})=\{c\}$.
\end{lemma}
\begin{proof}
	The subdendrite $E$ is a so-called A-set by \cite[(3.4) p.~69]{Why1942}.
	Thus, by \cite[(3.31) p.~69]{Why1942}, the components of $D\setminus E$
	are open, 
	form a (finite or infinite countable) null family, and every component
	has a singleton boundary.
	Hence, $\Bd(C)=\{c\}$ for some $c\in D$; clearly, $c\in \Bd(E)$ and, by Lemma~\ref{L:endS}, $c\in \End(\overline{C})$.

	Now assume that $E$ is nondegenerate. Then $D\setminus\{c\}$ has at least two components: one containing 
	$C$ and one intersecting the nonempty set $E\setminus\{c\}$. Thus $c$ is a cutpoint of $D$.	
	Trivially, $\Bd(\overline{C})\subseteq \Bd(C) = \{c\}$. On the other hand, $c\in \Bd(\overline{C})$
	because every neighbourhood of $C$ contains infinitely many points of the nondegenerate set $E$.
\end{proof}

The boundary of a subdendrite $E$ of a dendrite $D$ may be uncountable, even in the case when $D$ is completely regular. 
For example, let $f\colon [0,1]\to[0,1]$ have values $1/(n+1)$ at the endpoints 
of the contiguous intervals of range $n$ of the Cantor ternary set, 
and be zero otherwise.
Let $D$ be the \emph{subgraph} of $f$ (i.e., the set $\{[x,y]\in\RRR^2\colon x\in[0,1],\ 0\le y\le f(x)\}$)
and $E=[0,1]\times\{0\}$.
Then the boundary of $E$ in the space $D$ is the Cantor ternary set (multiplied by $\{0\}$).
However, we have at least the following.

\begin{lemma}\label{L:doplnkySubdendr-Hranica2}
	Let $D$ be a dendrite and $E$ be a proper subdendrite of $D$.
	Let $C_j$ ($j\in J$) be the components of $D\setminus E$.
	Put 
	\[
		\Bd\nolimits^*(E) = \{c_j\colon j\in J\},
		\qquad\text{where}\quad
		\{c_j\}=\Bd(C_j)
	\]
	and for every $c\in \Bd^*(E)$ put 
	\[
		J_c=\{j\in J\colon c_j=c\}
		\qquad\text{and}\qquad
		B_c=\{c\} \sqcup \bigsqcup_{j\in J_c} C_j
		= \bigcup_{j\in J_c} \overline{C_j}.
	\]
	Then $\Bd^*(E)$ is at most countable. It is a dense subset of $\Bd(E)$
	and every $B_c$ is a nondegenerate subdendrite of $D$ with
	$\Bd(B_c)=\{c\}$, so $B_c$ is regular closed. 
	Moreover,
	the subdendrites $B_c$ ($c\in\Bd^*(E)$) are the components
	of $\bigcup_{j\in J} \overline{C_j}$ and 
	$d(B_c,B_{c'})>0$ for every $c\ne c'$.\footnote{Examples show that the obvious inclusion
		$\bigcup\{B_c\colon c\in \Bd^*(E)\} =\bigcup_{j\in J} \overline{C_j} \subseteq \overline{D\setminus E}$
		is in general strict.}
\end{lemma}
A point $x$ of $\Bd(E)$ belongs to $\Bd^*(E)$ if and only if one can ``escape from $E$ through $x$'',
since a component of $D\setminus E$ is ``attached'' to it.
Therefore $\Bd^*(E)$ is said to be the \emph{escape-boundary} of $E$ in $D$.
\begin{proof}
Since the sets $C_j$ are disjoint and open in $D$, the set $J$ (hence also  $\Bd^*(E)$)
is at most countable. The union of $\overline{C_j}$ meets $\Bd(E)$ in a dense set. Hence
$\Bd^*(E)$ is a dense subset of $\Bd(E)$.

Clearly, every $B_c$ is connected, nondegenerate and $\Bd(B_c)=\{c\}\subseteq B_c$, thus $B_c$ is 
a nondegenerate subdendrite. 
Further, the subdendrites $B_c$ are obviously pairwise disjoint, hence $d(B_c,B_{c'})>0$
for every distinct $c,c'\in\Bd^*(E)$.

Put $\tilde{C}=\bigcup_{j\in J} \overline{C_j}$ and choose any $c\in\Bd^*(E)$.
Let $L$ be the component of $\tilde{C}$ intersecting (hence containing) $B_c$. 
Suppose that there is $x\in L\setminus B_c$. Let $j'\in J$ be such that $x\in \overline{C_{j'}}$;
put $c'=c_{j'}$.
Then $\overline{C_{j'}}\subseteq L$ and so $c'\in L$. Hereditary unicoherence of $D$ 
implies that $[c,c']\subseteq L\cap E$.
Since $\tilde{C}\cap E=\Bd^*(E)$ is countable, we have $c'=c$ and 
so $x\in B_c$. This contradicts the choice of $x$ and so we have that $L=B_c$, i.e., $B_c$ is
a component of $\tilde{C}$.
\end{proof}

\begin{lemma}\label{L:doplnkyBodov}
Let $D$ be a dendrite and $c_1,c_2\in D$. Let $C_1$ and $C_2$ 
be components
of $D\setminus\{c_1\}$ and $D\setminus\{c_2\}$, respectively.
Then exactly one of the following conditions holds:
\begin{enumerate}
\item\label{EQ:c1c2-a} 
	$c_1=c_2$ and $C_1=C_2$;
\item\label{EQ:c1c2-b} 
	$c_1=c_2$ and $C_1\cap C_2=\emptyset$;
\item\label{EQ:c1c2-c} 
	$c_1\ne c_2$, $c_1\in C_2$ and $c_2\in C_1$; in this case,
	$C_1\cap C_2\supseteq (c_1,c_2)$ is a nonempty proper subset of 
	both $C_1$ and $C_2$;
\item\label{EQ:c1c2-d} 
	$c_1\ne c_2$, $c_1\notin C_2$ and $c_2\notin C_1$; in this case,
	$C_1\cap C_2=\emptyset$;
\item\label{EQ:c1c2-e} 
	$c_1\ne c_2$, $c_1\notin C_2$ and $c_2\in C_1$; in this case,
	$C_2\subsetneq C_1$;
\item\label{EQ:c1c2-f} 
	$c_1\ne c_2$, $c_1\in C_2$ and $c_2\notin C_1$; in this case,
	$C_1\subsetneq C_2$.
\end{enumerate}
\end{lemma}
\begin{proof}
Recall that $C_i$ is an open connected set with boundary $\{c_i\}$, $i=1,2$.
It is trivial that if $c_1=c_2$ then either \eqref{EQ:c1c2-a} or \eqref{EQ:c1c2-b} is true. 

Let the assumptions in \eqref{EQ:c1c2-c} be fulfilled.
Then, for $i=1,2$, $\overline{C_i}$ contains both $c_1$ and $c_2$, hence
contains the arc $[c_1,c_2]$. Thus ${C_1}\cap {C_2}\supseteq (c_1,c_2)$.
Since $c_1\in C_2\setminus C_1$ and $c_2\in C_1\setminus C_2$,
the intersection ${C_1}\cap {C_2}$ is a proper subset of 
both $C_1$ and $C_2$.

If we are in \eqref{EQ:c1c2-d}, suppose, on the contrary, that
there is $d\in C_1\cap C_2$. Then $[d,c_i)\subseteq C_i$
for $i=1,2$. The union of the arcs $[c_1,d]$ and $[d,c_2]$
is a path from $c_1$ to $c_2$.
Replacing $d$ by another point from the intersection of these arcs, if necessary,
we may assume that the mentioned path is an arc.
Since dendrites are uniquely arcwise connected, this arc containing $d$ coincides with
the arc $[c_1,c_2]$. It follows that $d\in [c_1,c_2]$.
Then $(c_1,c_2]$ is a connected set intersecting $C_1$ and not intersecting the boundary of $C_1$. Now Lemma~\ref{L:boundaryBumping}
yields that $(c_1,c_2]\subseteq C_1$, a contradiction
with $c_2\notin C_1$.

Since \eqref{EQ:c1c2-f} is analogous to \eqref{EQ:c1c2-e},
it suffices to prove \eqref{EQ:c1c2-e}. So assume that 
$c_1\ne c_2$, $c_1\notin C_2$ and $c_2\in C_1$.
Since $C_1$ is a neighbourhood of $c_2$ and $c_2\in\overline{C_2}$,
$C_2$ intersects $C_1$ but does not intersect the boundary of $C_1$.
By Lemma~\ref{L:boundaryBumping}, $C_2\subseteq C_1$. Clearly $C_2\ne C_1$. 
\end{proof}

\begin{lemma}\label{L:doplnkySubdendr}
	Let $D$ be a dendrite, $A_i,B_i$ ($i=1,2$) be nondegenerate subdendrites of $D$ such that
	$A_i\cap B_i=\emptyset$ for $i\in\{1,2\}$ and the intersections $A_1\cap A_2$, $B_1\cap B_2$
	and $A_1\cap B_2$ are nonempty.
    Denote by $C_i$ the (unique) component of $D\setminus B_i$ containing $A_i$.
	Then 
	\begin{equation}\label{EQ:BdCi}
		\Bd(C_i)=\Bd(\overline{C_i})  = \{c_i\}
		\qquad\text{for some } c_i\in B_i\cap \Cut(D)
		\quad(i\in\{1,2\}).
	\end{equation}
	Moreover, $c_2\in A_1\cap B_2$, $c_1\notin \overline{C_2}$ and $C_2\subsetneq C_1$.
\end{lemma}
\begin{proof}
The claim \eqref{EQ:BdCi} follows from Lemma~\ref{L:doplnkySubdendr-Hranica}.
Further, by the assumptions, $C_1\cap C_2\supseteq A_1\cap A_2\ne\emptyset$
and $C_1\setminus C_2\supseteq A_1\cap B_2\ne\emptyset$.
Thus we have either \eqref{EQ:c1c2-c} or \eqref{EQ:c1c2-e} from Lemma~\ref{L:doplnkyBodov}.
In particular, $c_1\ne c_2$.

We show that \eqref{EQ:c1c2-c} is also impossible.
Suppose, on the contrary, that $c_1\in C_2$ and $c_2\in C_1$. Then, by 
Lemma~\ref{L:doplnkyBodov}\eqref{EQ:c1c2-c}, 
$(c_1,c_2)\subseteq C_1\cap C_2$
and so $(c_1,c_2)\cap (B_1\cup B_2)=\emptyset$.
Choose any $b\in B_1\cap B_2$; clearly, $b\notin C_1\cup C_2$. 
By \eqref{EQ:BdCi},
$c_i\in B_i$ and hence $[b,c_i]\subseteq B_i$
for $i\in\{1,2\}$.
Clearly, $[b,c_1]\cup [c_1,c_2]\cup [c_2,b]$ contains
a simple closed curve
(because the first and the third arcs are in $B_1\cup B_2$, while the second arc (minus the endpoints) 
is disjoint from $B_1\cup B_2$), a contradiction.

We have showed that \eqref{EQ:c1c2-e} from Lemma~\ref{L:doplnkyBodov}
is true. Thus $c_2\in C_1$, $c_1\notin {C_2}$ and $C_2\subsetneq C_1$.
Since also $c_2\ne c_1$ and $\Bd(C_2)=\{c_2\}$, we have that 
$c_1\notin \overline{C_2}$.
Since $c_2\in B_2$, it remains to show that $c_2\in A_1$.
Since $A_1$ is connected and intersects both $B_2$ and the superset $C_2$ of $A_2$,
Lemma~\ref{L:boundaryBumping} gives that $A_1$ intersects $\Bd(C_2)=\{c_2\}$, thus $c_2\in A_1$.
\end{proof}

For a metric space
$(X,d)$, the Hausdorff one-dimensional measure is denoted by $\HHh^1_d$;
if $(X,d)$ is a closed real interval with the Euclidean metric, we write simply
$\abs{\cdot}$ instead of $\HHh^1_d(\cdot)$.

By \cite{Har44} (see also \cite{BNT92}), every dendrite $D$ admits a convex metric $d$
such that $D$ has \emph{finite length} with respect to this metric, i.e., $\HHh^1_d(D)$ is finite. 
Here \emph{convex} means that for any $x,y \in D$ there exists $z\in D$ such that $d(x,z) = d(z,y) = d(x,y)/2$.

\begin{lemma}\label{L:convexMetric1}
Let $D$ be a nondegenerate dendrite with a convex metric $d$. Then,
\begin{enumerate}
	\item\label{ENUM:convex-H1}
	for every distinct $x,y\in D$, the arc $[x,y]$ is a \emph{geodesic} in the sense that $\HHh^1_d([x,y]) = d(x,y)$;
	\item\label{ENUM:convex-xyz} $d(x,y)=d(x,z)+d(z,y)$ for every distinct $x,y\in D$ and every $z\in [x,y]$.
\end{enumerate}	
\end{lemma}
\begin{proof}
For \eqref{ENUM:convex-H1} see \cite{Eil44}. Then \eqref{ENUM:convex-xyz} follows from \eqref{ENUM:convex-H1}
(alternatively, one can start by a repeated use of the definition of a convex metric).
\end{proof}

Let $E$ be a subdendrite of a dendrite $D$ and let $r\colon D\to E$ be the first point map, see
\cite[Theorem~10.26]{Nad92}. If $x\in D$ then $r(x)$ will be called the \emph{projection} of $x$ into $E$
and denoted by $\proj(x,E)$.

\begin{lemma}\label{L:convexMetric2}
Let $D$ be a nondegenerate dendrite with a convex metric $d$. Let $E$ be a subdendrite of $D$ and $x\in D$.
\begin{enumerate}
\item\label{ENUM:convex-arc_xe} If $x\notin E$ then,
for every $e\in E$, 
	\[
	\proj(x,E)\in [x,e]
	\qquad\text{and}\qquad
	[x,e]\cap E = [\proj(x,E),e].
	\]
\item\label{ENUM:convex-uniqueMin}
	The projection $\proj(x,E)$ is the unique point $e_0\in E$ such that
	\[
	d(x,E) = d(x,e_0).
	\]
\item\label{ENUM:convex-cutpoint} If $x\notin E$ and
	$E$ is nondegenerate then $\proj(x,E)$ is a cutpoint of $D$.
\item\label{ENUM:convex-singletonBoundary} If $x\notin E$ then $\proj(x,E)\in\Bd(E)$.
\end{enumerate}	
\end{lemma}
\begin{proof}
\eqref{ENUM:convex-arc_xe} For the first part see \cite[Lemma~10.24]{Nad92}. The intersection
$[x,e]\cap E$ is a subcontinuum of $D$ containing both $\proj(x,E)$ and $e$.
If it is a singleton then $\proj(x,E)=e$ and we are done. Otherwise
it is a subarc $[e_1,e]$ of $[x,e]$ with $e_1\in E\cap[x,e]$. By the first part, $\proj(x,E)\in [x,e_1]$.
Since $[x,e_1]\cap E=([x,e_1]\cap [x,e])\cap E=[x,e_1]\cap [e_1,e]=\{e_1\}$
we get $e_1=\proj(x,E)$ and we are done.

\eqref{ENUM:convex-uniqueMin} By \eqref{ENUM:convex-arc_xe} and convexity,
$d(x,\proj(x,E)) = d(x,E)$.
Take any $e\in E$, $e\ne \proj(x,E)$. Then, by Lemma~\ref{L:convexMetric1}\eqref{ENUM:convex-xyz},
$d(x,e)=d(x,\proj(x,E)) + d(\proj(x,E),e) > d(x,E)$. 

\eqref{ENUM:convex-cutpoint} 
Choose a point $e\in E$, $e\ne\proj(x,E)$. By \eqref{ENUM:convex-arc_xe}, $\proj(x,E)\in (x,e)$.
Hence $\proj(x,E)$ is a cutpoint of $D$.

\eqref{ENUM:convex-singletonBoundary} This follows from 
\eqref{ENUM:convex-arc_xe} and Lemma~\ref{L:boundaryBumping}.
\end{proof}

\begin{lemma}\label{L:finiteLength}
Let $D$ be a dendrite of finite length and $f\colon D\to D$ be a continuous map. 
Let $E\subseteq D$ be a connected set with $\limsup_{n\to\infty} \diam f^n(E)>0$.
\begin{enumerate}
\item\label{ENUM:fnEfnkE} There exists a least nonnegative integer $n_0$ such that, for some positive integer $k$,
	\[
	f^{n_0}(E)\cap f^{n_0+k}(E)\ne\emptyset.
	\]
\item\label{ENUM:fnE-Ki} 
    Let $k$ be any positive integer with that property. Then the sets
	\[
  		K_i=\bigcup_{j=0}^{\infty}f^{n_0+i+jk}(E)
  		\qquad
  		\text{for}\quad i\in\{0,1,\dots,k-1 \},
	\]
	are connected subsets of $D$ such that 
	\begin{equation}\label{EQ:cycPerm}
		f(K_i)=K_{i+1}
		\quad\text{for }i\in\{0,1,\dots, k-2\}
		\qquad\text{and}\qquad 
		f(K_{k-1})\subseteq K_0.
	\end{equation}
\item\label{ENUM:fnE-Lj}  
    The orbit $\Orb_f(f^{n_0}(E))$ has components $L_0$ ($\supseteq K_0\supseteq f^{n_0}(E)$), $L_1,\dots,L_{r-1}$, 
    where $r$ is a divisor of $k$, 
	\begin{equation}\label{EQ:cycPermL}
		f(L_j)=L_{j+1}
		\quad\text{for }j\in\{0,1,\dots, r-2\}
		\qquad\text{and}\qquad 
		f(L_{r-1})\subseteq L_0.
	\end{equation}
	Moreover, for every $j\in\{0,1,\dots,r-1\}$,
	$
		L_j = \bigcup_{\ell=0}^{k/r-1} K_{j+\ell r}.
	$
\item\label{ENUM:fnE-Orb} 
    $
    \Orb\nolimits_{f}(E)=E\sqcup f(E)\sqcup \dots \sqcup f^{n_0-1}(E) \sqcup L_0\sqcup L_1\sqcup\dots\sqcup L_{r-1}.
    $
\end{enumerate}
\end{lemma}
\begin{proof}
\eqref{ENUM:fnEfnkE} By the assumption, there are $\eps>0$ and positive integers 
$n_1<n_2<\dots$ such that $\diam f^{n_i}(E)>\eps$ for every $i$. For every $i$ choose $a_i,b_i\in f^{n_i}(E)$
with $d(a_i,b_i)>\eps$. By connectedness, the arc $A_i=[a_i,b_i]$ is a subset of $f^{n_i}(E)$.
Thus $\HHh^1_d(A_i)>\eps$ for every $i$. Since the length of $D$ is finite, the arcs $A_i$ cannot be disjoint,
hence \eqref{ENUM:fnEfnkE}.

\eqref{ENUM:fnE-Ki} The properties of the sets $K_i$ are obvious. 

\eqref{ENUM:fnE-Lj}
It follows that 
$\Orb_f(f^{n_0}(E))=\bigcup_{i=0}^{k-1} K_i$
has at most $k$ components;
let $L_0,L_1,\dots,L_{r-1}$ be the list of them. Clearly, every $L_j$ is the union of some of the connected 
sets $K_i$.
Each of the components $L_j$ is mapped to a component. 
Less than $r$ components cannot form a cycle since every point from every 
$K_i$ visits every $K_j$ repeatedly. 
Therefore, with appropriate notation, we have \eqref{EQ:cycPermL}.
Choose a point $x_0\in L_0$. It comes back to $L_0$ only in times which are multiples of $r$. 
However, it belongs to some $K_i\subseteq L_0$ and so it comes back to $L_0$ also in the time $k$
by \eqref{EQ:cycPerm}.
Hence $k$ is a multiple of $r$.

We may assume that $L_0\supseteq K_0$ and $f(L_j)\subseteq L_{j+1 \operatorname{mod} r}$.
Then clearly $L_j \supseteq \bigcup_{\ell=0}^{k/r-1} K_{j+\ell r}$ for all $j$.
Since we have used all the sets $K_i$ here, these inclusions are in fact equalities.

\eqref{ENUM:fnE-Orb} This follows from \eqref{ENUM:fnE-Lj} and definition of $n_0$.
\end{proof}

The above three lemmas indicate that it will be useful to adopt the following convention
(recall that it is possible due to \cite{Har44,BNT92}).

\medskip

\begin{quotation}
	\noindent {\bf Convention.} 
	From now on, till the end of Section~\ref{S:thm21},
	we will always assume that the metric $d$ on a dendrite $D$ is convex and such that $D$ has finite length with respect to this metric.
\end{quotation}

\medskip

\noindent
It is important to realize that, when proving Theorem~\ref{T:dendrite_gch=gech},
we are allowed to adopt this convention,
i.e., to replace the original metric on $D$ by an equivalent metric from the convention.
Indeed, the notions of ``generic chaos'' and ``generic $\eps$-chaos for some $\eps>0$'' 
are invariants of topological conjugacy on compact metric spaces (this is a simple consequence
of uniform continuity of the conjugating homeomorphism as was observed already in \cite{Sno1990}).

We will also use the simple fact that if $f$ is a continuous selfmap of a compact metric space
and $g=f^k$ for some positive integer $k$,
then $f$ is generically chaotic (generically $\eps$-chaotic) if and only if so is $g$.

\smallskip

Since completely regular dendrites appear in Theorem~\ref{T:dendrite_gch=gech}, the following simple observation
will be useful (and repeatedly used).  
To state it, first recall the definition of the Riemann dendrite.
Let $r:[0,1] \to \mathbb R$ be the Riemann function
(called also Thomae function), i.e., the function defined by $r(x)=0$ if $x$ is irrational and $r(x) =1/q$ if $x=p/q$ where $q$ is a positive integer, $p$ is a nonnegative integer and $p$ and $q$ are relatively prime. Then the subgraph of $r$ is a dendrite; we call it the \emph{Riemann dendrite}, see Figure~\ref{fig:Rieman}.

\begin{proposition}\label{P:conditions}
	Let $D$ be a dendrite. Then the following six conditions are equivalent.
	\begin{enumerate}
		\item $D$ is completely regular (i.e., every nondegenerate subdendrite of $D$ has nonempty interior).
		\item Every nondegenerate subdendrite of $D$ is a regular closed set.
		\item Every arc in $D$ has nonempty interior in $D$.
		\item\label{ENUM:complRegFreeArc} Every arc in $D$ contains a subarc which is a free arc in $D$.
		\item There is no arc $A$ in $D$ such that the set $A\cap B(D)$ is dense in $A$.
		\item $D$ does not contain a copy of the Riemann dendrite.
	\end{enumerate}
	Further, we have the implications
	$$
	\Branch(D) \text{ is discrete } \Longrightarrow D\text{ is completely regular} \Longrightarrow \Branch(D) \text{ is nowhere dense},
	$$
	while the converse implications do not hold. 
\end{proposition}

\begin{proof}
	The equivalence of the six conditions is obvious.
	
	Suppose that $\Branch(D)$ is discrete and $A$ is an arc in $D$. Then there is a subarc $A'$ of $A$ such that it does not contain any branch point of $D$. However, then $A'$ is a free arc in $D$ and we get \eqref{ENUM:complRegFreeArc}. 
	Now suppose that $D$ is completely regular and that $\Branch(D)$ is dense in an open set $\emptyset \neq U \subseteq D$. Fix an arc $A\subseteq U$. Since $\overline{\Branch(D)}\supseteq A$, $A$ does not contain any free arc, a contradiction with \eqref{ENUM:complRegFreeArc}. 
	
	The Riemann dendrite has a nowhere dense set of branch points but is not completely regular.
	
	Finally, let $f:[-1,1] \to \mathbb R$ be the function defined by $f(1/n) = 1/n$ for every $n\in \mathbb N$, $f(0)	=1$ and $f(x)=0$ for all other points $x\in [-1,1]$. Then the subgraph of $f$ is a completely regular dendrite whose set of branch points is not discrete, see Figure~\ref{fig:comb}.
\end{proof}

\begin{figure}[ht]
	\centering
	\includegraphics[width=.3\linewidth]{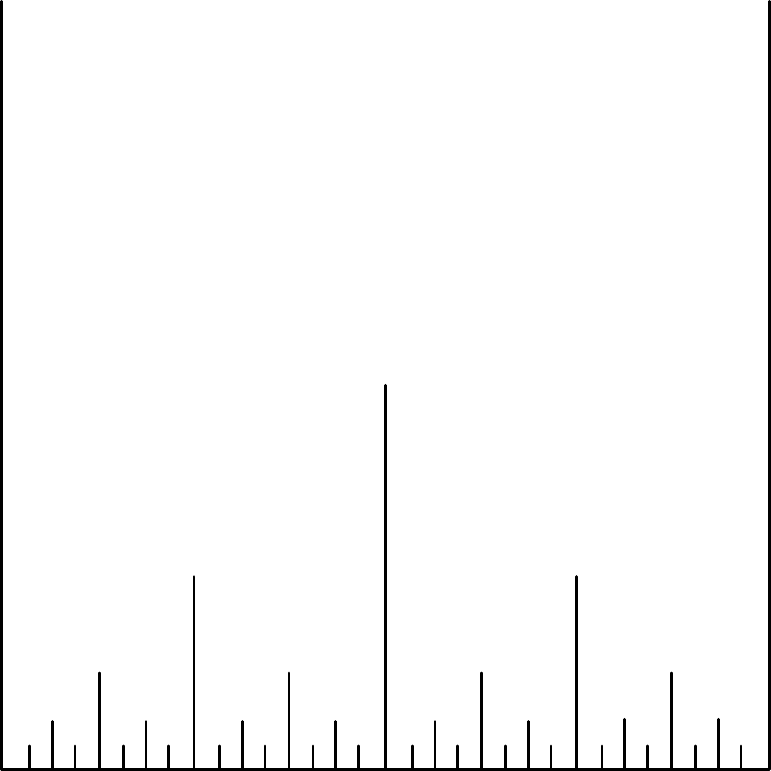}
	\caption{The Riemann dendrite has nowhere dense set of branch points but it is not completely regular.}
	\label{fig:Rieman}
\end{figure}

\begin{figure}[ht]
	\centering
	\includegraphics[width=.6\linewidth]{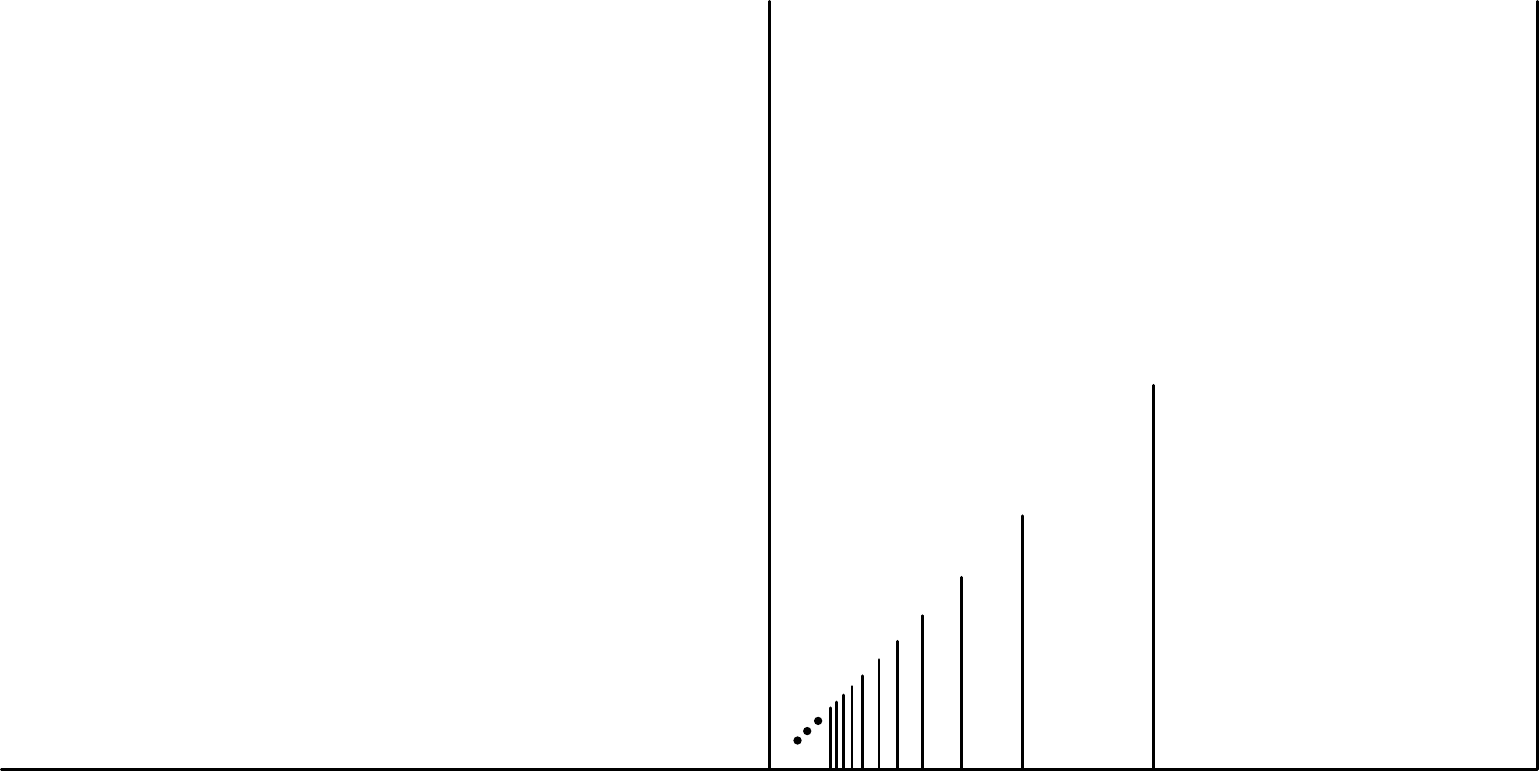}
	\caption{A completely regular dendrite whose set of branch points is not discrete.}
	\label{fig:comb}
\end{figure}

In general, generic chaos is not carried over to invariant subsets (for instance, a generically chaotic map on a dendrite may be equal to the identity on a subarc, see Proposition~\ref{P:exact} below). 
However, the following is true.

\begin{lemma}[Proposition~13 in \cite{Tak2016}]\label{L:zuzenieGCH}
	Let $X$ be a compact metric space, $f\colon X\to X$ be continuous and $Y$ be a regular closed
	$f$-invariant subset of $X$. If $f$ is generically chaotic or generically $\eps$-chaotic, then so is 
	$f|_Y\colon Y\to Y$.
\end{lemma}

Clearly, a singleton does not admit a generically chaotic selfmap. Therefore, in auxiliary results
in the next sections, 
we will always assume that the phase space is a \emph{nondegenerate} dendrite.

\section{Fixed points of dendrite maps}\label{S:fixed-points}

Let $D$ be a nondegenerate dendrite and $a,b,x\in D$. We say that $x$ \emph{separates} $a$ and $b$ if $a$ and $b$ lie in different components of $D\setminus \{x\}$.
For $a\in D$ let $\leq_a$ be the partial order on $D$ defined by
$x\leq_a y$ whenever $x\in [a,y]$; if $x\leq_a y$ and $x\ne y$ we will write $x<_a y$. 
If $E$ is a subdendrite of $D$
and $a\in D\setminus E$ then, by Lemma~\ref{L:convexMetric2}\eqref{ENUM:convex-arc_xe}
(recall Convention),
\[
  \proj(a,E) = \inf\nolimits_{\leq_a} E.
\]

For any $x\in D$ the set
\begin{equation}\label{Eq:Max}
D^a(x)=\{y\in D\colon x\leq_a y\}
\end{equation} 
is a subdendrite of $D$. Notice that if $x\ne a$ then $D^a(x)$ is the set containing $x$ and all those points of $D$ which are separated from $a$ by $x$. Since the metric $d$ is convex by Convention, if $y \in D^a(x)$ then $d(y,a) \geq d(x,a)$ and so $x=\proj(a,D^a(x))$.

\begin{figure}[ht]
    \centering
    \includegraphics[width=0.5\textwidth]{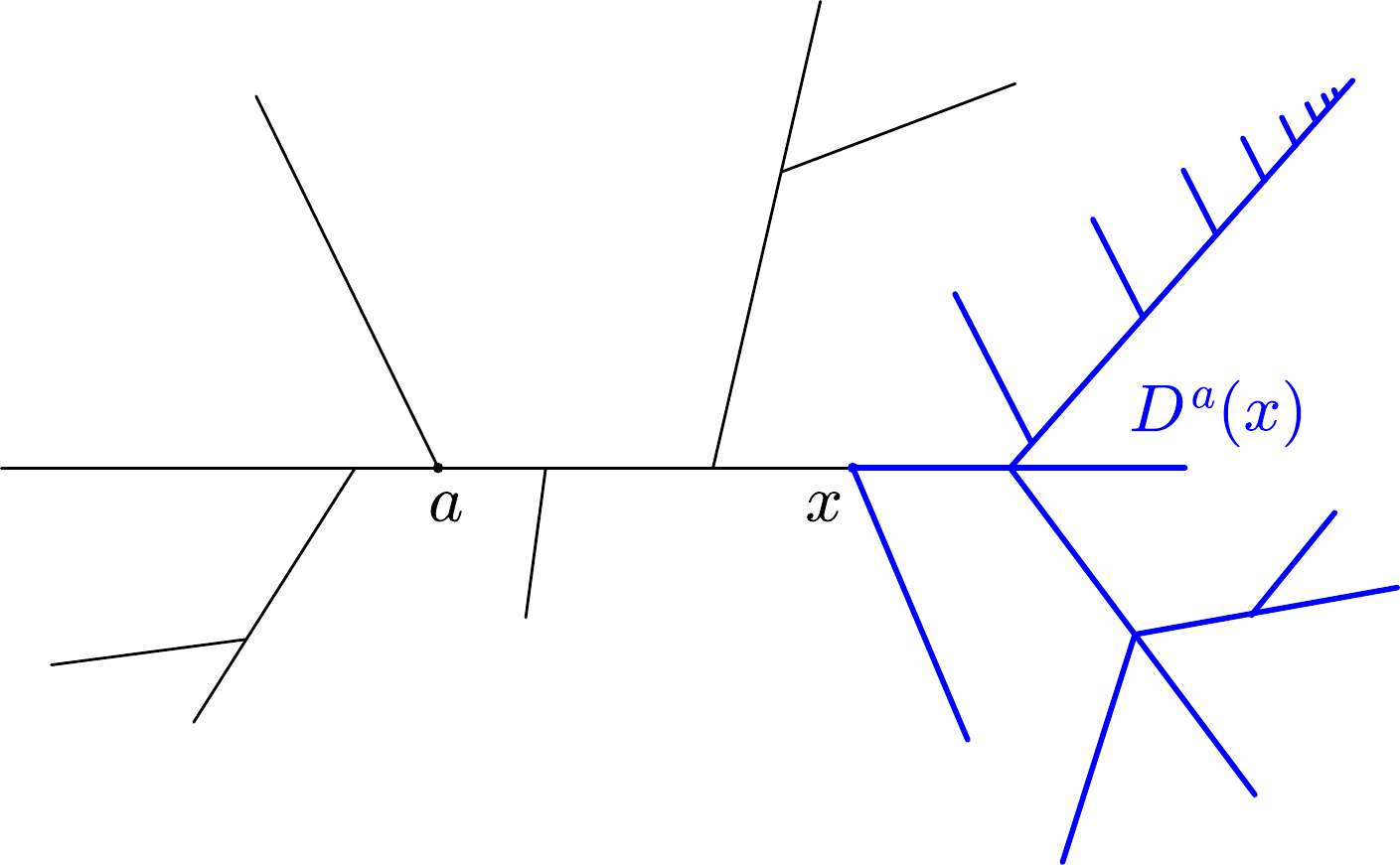}
    \caption{The set $D^a(x)$.}
    \label{fig:Max}
\end{figure}

For any distinct $a,b\in D$ let us define a subdendrite $D_{[a,b]}\subseteq D$ by
\begin{equation}\label{Eq:Dab}
D_{[a,b]}=\overline{D\setminus\left({D^b(a)\cup D^a(b)}\right)}.
\end{equation}
One can imagine it as the subdendrite ``enclosed by $a$ and $b$''. Put $D_{(a,b]}=D_{[a,b]}\setminus\{a\}$, $D_{[a,b)}=D_{[a,b]}\setminus\{b\}$ and $D_{(a,b)}=D_{[a,b]}\setminus\{a,b\}$. 
Let $D_{[a,a]}$ denote the singleton $\{a\}$. 

Now consider a continuous map $f\colon D\to D$ on a nondegenerate dendrite $D$.
Note that, for $a\neq x$ in $D$, there are four mutually exclusive possibilities: 
\begin{itemize}
\item $f(x)=x$, i.e., $x$ is a fixed point of $f$;
\item $f(x)\in D^a(x)\setminus\{x\}$, in this case we say that $x$ \emph{evades} $a$, see Figure~\ref{fig:test2};
\item $f(x)\in D_{[a,x)}$, in this case we say that $x$ \emph{admires} $a$, see Figure~\ref{fig:test1};
\item $a$ separates $x$ from $f(x)$, in this case we say that $x$ \emph{jumps over} $a$.
\end{itemize}
If $a$ is an endpoint of $D$, then the fourth possibility cannot occur and so we have a trichotomy for the points $a\neq x$: either $x$ is fixed or $x$ evades $a$ or $x$ admires $a$. In particular, if also $x$ is an endpoint of $D$, then we have only a dichotomy: either $x$ is fixed or $x$ admires $a$.

\begin{figure}[ht]
\centering
\begin{minipage}{.49\textwidth}
  \centering
  \includegraphics[width=.88\linewidth]{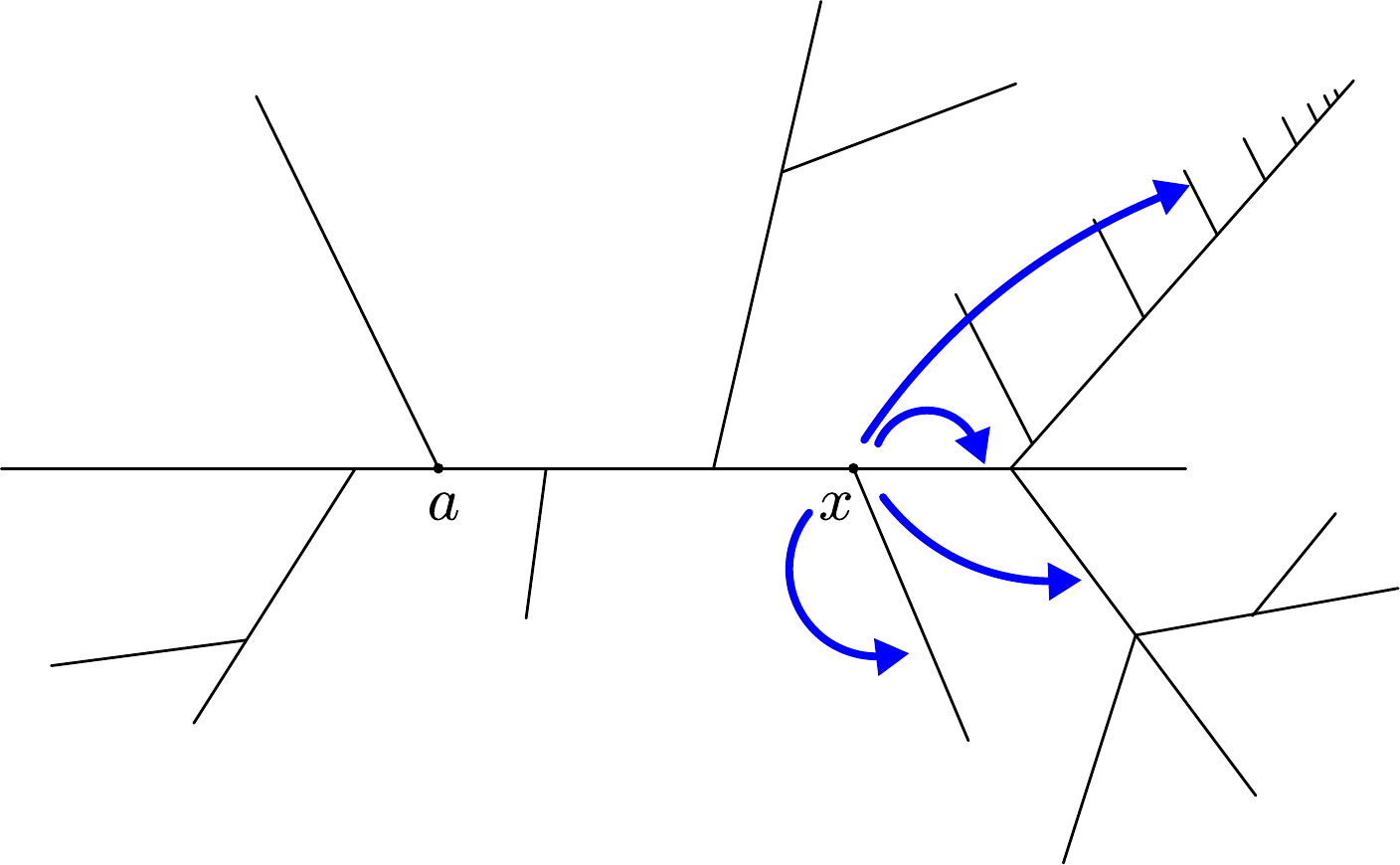}
  \caption{$x$ evades $a$.}
  \label{fig:test2}
\end{minipage}
\begin{minipage}{.49\textwidth}
	\centering
	\includegraphics[width=.88\linewidth]{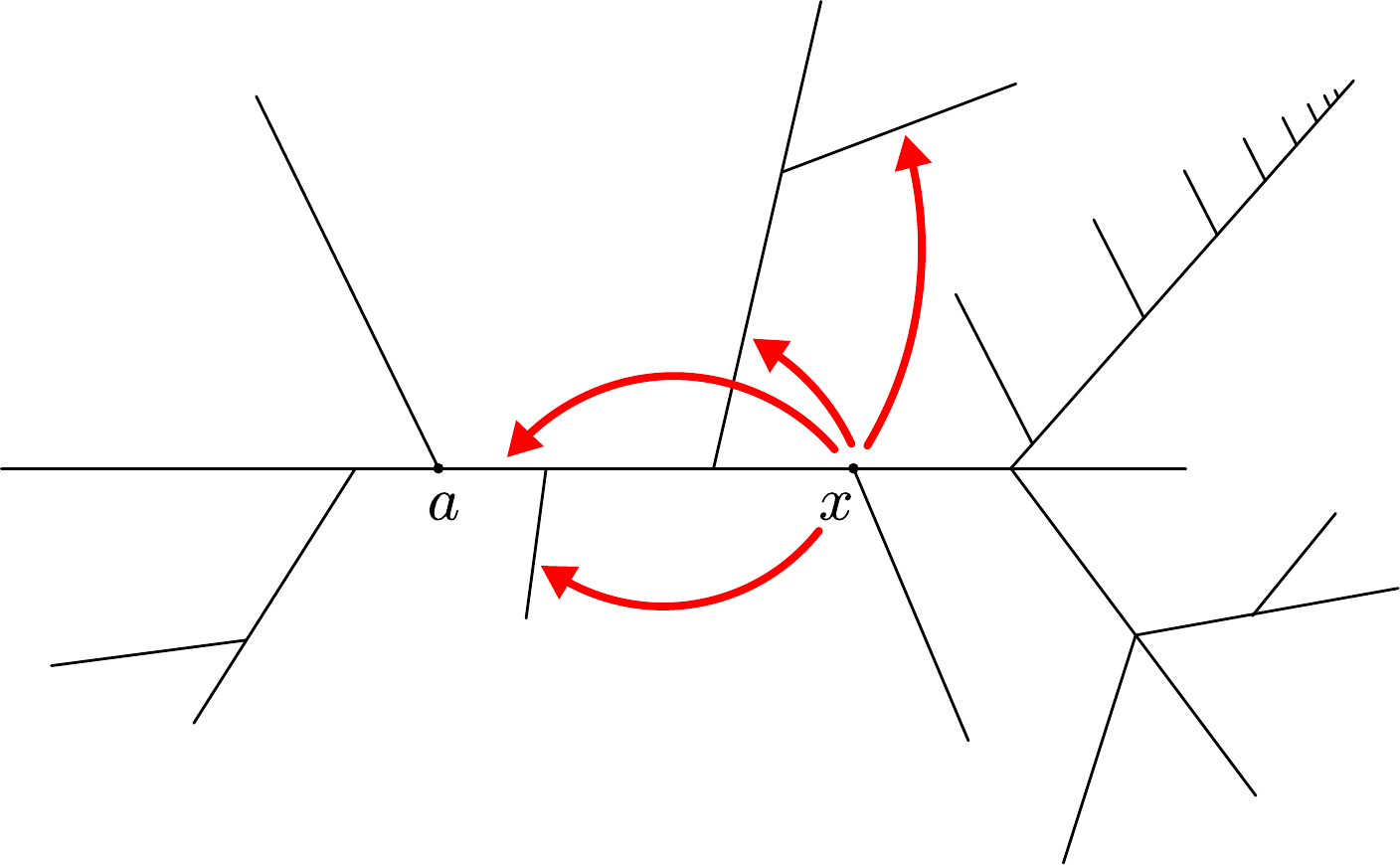}
	\caption{$x$ admires $a$.}
	\label{fig:test1}
\end{minipage}
\end{figure}

\begin{definition}[see Definition~5.3.2 in \cite{Blo2013}]\label{D:weakly cut}
Let $D$ be a nondegenerate dendrite and $f\colon D\to D$ be continuous. Let $a\in D$ be a fixed point of $f$ and $B$ be a component of $D\setminus \{a\}$. Then $a$ is called a \emph{weakly repelling fixed point of $f$ for $B$} if at least one of the following two conditions holds:
\begin{itemize}
	\item in $B$, arbitrarily close to $a$ there is a cutpoint of $D$ which is fixed by $f$, or 
	\item in $B$, arbitrarily close to $a$ there is a cutpoint $x$ separating $a$ from $f(x)$ (i.e., $x$ evades $a$).
\end{itemize}
\end{definition}

We omit the obvious proof of the following lemma.

\begin{lemma}\label{L:cutpoint_attracting_admires}
Let $D$ be a nondegenerate dendrite and $f\colon D\to D$ be continuous. Let $a\in D$ be a fixed point of $f$ and $B$ be a component of $D\setminus \{a\}$. Then $a$ is \emph{not} weakly repelling for $B$ if and only if there exists a neighbourhood $U$ of $a$ in $D$ such that, for every cutpoint $x\in U\cap B$, 
\begin{itemize}
	\item either $x$ admires $a$,
	\item or $x$ jumps over $a$.
\end{itemize} 	
\end{lemma}

The following fact is well known, see \cite[Lemma~3.4]{Sch1978} or \cite[Theorem~7.2.2(1)]{Blo2013}.
\begin{lemma}\label{L:continuity-implies-fixed-pt}
Let $D$ be a nondegenerate dendrite and $f\colon D\to D$ be continuous. If $x,y\in D$ are such that 
$x$ evades $y$ and $y$ evades $x$,
then $f$ has a fixed point in $(x,y)$.
\end{lemma}

\begin{lemma}[see Lemma 7.2.5 in \cite{Blo2013}]\label{L:noneffluent_endpoint}
Let $D$ be a nondegenerate dendrite and $f\colon D\to D$ be continuous. Then 
\begin{itemize}
\item either there is
	a fixed point of $f$ which is a cutpoint of $D$,
\item or there is a fixed point $a\in D$ which is an endpoint of $D$ and is not weakly repelling of $f$ 
    for $D\setminus\{a\}$.
\end{itemize}
\end{lemma}

\begin{lemma}\label{L:noneffluent_endpoint-every_point_admires}
Let $D$ be a nondegenerate dendrite and $f\colon D\to D$ be continuous such that $\Fix(f)\subseteq\End(D)$. 
Then there exists an endpoint $a$, fixed by $f$ and such that every $y\in D$ which is not fixed by $f$ admires $a$.
Such an endpoint $a$ is unique.
\end{lemma}

\begin{proof}
Lemma~\ref{L:noneffluent_endpoint} guarantees the existence of an endpoint $a$ fixed by $f$ which is not a weakly repelling point of $f$ for $D\setminus\{a\}$. Then, by Lemma~\ref{L:cutpoint_attracting_admires}, there exists a neighbourhood $U$ of $a$ such that every cutpoint in $U$ admires $a$. We need to prove that, in fact, all points in $D\setminus \{a\}$ which are not fixed by $f$ admire $a$. Suppose, on the contrary, that some point $w\in D\setminus\{a\}$ evades $a$. So $w$ is a cutpoint because endpoints never evade $a$ and then $w$ is in $D\setminus U$. Consider any cutpoint 
$z\in U\cap (a,w)$. We know that $z$ admires $a$, hence $z$ evades $w$. Clearly, $w$ evades $z$.
Then, by Lemma~\ref{L:continuity-implies-fixed-pt}, $f$ has a fixed point in $(z,w)$, which is obviously a cutpoint, a contradiction.

Suppose that there are two distinct endpoints $a_1,a_2$ with the property from the lemma.
Choose a point $x\in(a_1,a_2)$. Since $x\notin\End(D)$, it is admired both by $a_1$ and $a_2$, i.e.,
$f(x)\in D_{[a_1,x)} \cap D_{[a_2,x)}$. However, this intersection
is obviously empty, a contradiction.
\end{proof}

\begin{lemma}\label{L:at_most_one_WR_point}
Let $D$ be a nondegenerate dendrite and $f\colon D\to D$ be continuous. Suppose that $p$ and $q$ are two distinct endpoints fixed by $f$, which are not weakly repelling of $f$ for $D\setminus\{p\}$ and $D\setminus\{q\}$, respectively. Then there exists a cutpoint $r\in (p,q)$ fixed by $f$.
\end{lemma}

\begin{proof}
By Lemma~\ref{L:cutpoint_attracting_admires} there exist disjoint neighbourhoods $U_p$ and $U_q$ of $p$ and $q$, respectively, such that every cutpoint in $U_p$ admires $p$ and every cutpoint in $U_q$ admires $q$. We can suppose that $\diam U_p<d(p,q)/2$ and $\diam U_q<d(p,q)/2$. Fix $p^*\in U_p\cap (p,q)$ and $q^*\in U_q\cap (p,q)$. Clearly, $p^*$ and $q^*$ are cutpoints and so they admire $p$ and $q$, respectively. As the diameters of $U_p$ and $U_q$ are less than $d(p,q)/2$, the points $p,p^*,q^*$ and $q$ are positioned on $[p,q]$ in this order. The fact that $p^*$ admires $p$ gives us $f(p^*)\in D_{[p,p^*)}$. Similarly, $f(q^*)\in D_{(q^*,q]}$. Then, by Lemma~\ref{L:continuity-implies-fixed-pt} there is a fixed point $r\in(p^*,q^*)\subseteq (p,q)$. Obviously, $r$ is a cutpoint.
\end{proof}


\section{Generic chaos on completely regular dendrites}\label{S:gen-chaos}

The first lemma in this section does not need the assumption of complete regularity.

\begin{lemma}\label{L:fixed_pt_in_intersection}
Let $D$ be a nondegenerate dendrite and $f\colon D\to D$ be continuous such that the condition \ref{ENUM:B1} is satisfied
by the family of nondegenerate subdendrites. 
Then the intersection of any nonempty family of invariant nondegenerate subdendrites
is an invariant (possibly degenerate) subdendrite and hence it contains a fixed point of $f$.
\end{lemma}

\begin{proof}
The intersection of a family of invariant nondegenerate subdendrites is trivially invariant and closed.
Using unicoherence of $D$ it is a subdendrite if it is nonempty. Hence it is sufficient to prove that 
the intersection is nonempty. Since $D$ is compact it is in fact sufficient to prove that 
any family of invariant nondegenerate subdendrites of $D$ has the finite intersection property. 
So we need to show that any finite family of invariant nondegenerate subdendrites 
has nonempty intersection. In view of Lemma~\ref{L:finiteCap} it is sufficient to prove that
if $D_1$ and $D_2$ are two invariant nondegenerate subdendrites then $D_1\cap D_2\ne\emptyset$.
However this is obvious because, by \ref{ENUM:B1}, $d(D_1,D_2)=0$ and since $D_1$ and $D_2$ are compact sets, we have $D_1\cap D_2\neq \emptyset$.
\end{proof}

A generically chaotic map on an $\omega$-star may have arbitrarily small invariant nondegenerate subdendrites \cite{Mur00}. The next two lemmas show that on some dendrites this cannot happen.

\begin{lemma} \label{L:sub_cont_fix_pt-diam_delta}
Let $D$ be a completely regular nondegenerate dendrite with all points of finite order. Let $f\colon D\to D$ be generically chaotic and $p\in D$ be a fixed point of $f$. Then there exists $\delta>0$ such that every invariant nondegenerate subdendrite $S\subseteq D$ containing $p$ has diameter at least $\delta$. 
\end{lemma}

\begin{proof}
Suppose, on the contrary, that 
for every $\delta>0$ there exists an invariant nondegenerate subdendrite $S(\delta)\subseteq D$ such that $p\in S(\delta)$ and $\diam S(\delta)<\delta$.
To obtain a contradiction, we proceed in three steps.

\smallskip
\textit{Step~1.} We construct a nested sequence of invariant nondegenerate subdendrites $\{T_n\}_{n=0}^\infty$ 
containing $p$ and converging to $p$.
\smallskip

Let $\CCc$ be the set of components of $D\setminus\{p\}$. 
Since $\mathcal C$ is finite, 
there are a subsequence $\{S_n\}_{n=0}^\infty$ of $\{S(1/k)\}_{k=1}^\infty$ and 
a nonempty set $\mathcal B\subseteq \mathcal C$ with the following property:
For every $n$, the dendrite $S_n$ intersects a component of $D\setminus\{p\}$ if and only if 
this component belongs to $\BBb$. Clearly, $\diam S_n\to 0$.

For every $n$, $T_n=\bigcap_{i=0}^{n}S_i$ is an invariant nondegenerate subdendrite of $D$ 
and $T_{n+1}\subseteq T_n$. 
Every $T_n$ intersects  each component from $\mathcal B$ and no component from $\mathcal C\setminus \mathcal B$. 
Since $\diam T_n\to 0$, by passing to a subsequence if necessary, we may assume that $\diam T_{n+1}<\diam T_{n}$ for every $n$.

\smallskip
\textit{Step~2.} We show that, for every $N\in \NNN$, the set 
\begin{equation*}
A_N = \{x\in T_0\colon \Orb\nolimits_f(x)\cap T_N = \emptyset\}.
\end{equation*}
is nowhere dense in $T_0$.\footnote{\label{Footnote-dend}Thus, generic chaos on $D$ implies that most of the points from $T_0\setminus T_N$ enter $T_N$. This is not trivial. For instance let $D$ be the dendrite from Figure~\ref{fig:comb}
where the horizontal arc has endpoints $(-1,0)$ and $(1,0)$,
with $p=(0,0)$ being the leftmost branch point. Imagine that $T_0$ consists of $p$ and the points
$(a,b)\in D$ with $a>0$, and $T_n$ is the arc joining $p$ and $(1/(n+1),0)$, $n=1,2,\dots$. If for instance $(x,y)$ is a Li-Yorke pair with $x\in T_0\setminus T_1$ and $y\in T_2$, the point $x$ is proximal to the invariant set
$T_2$ but this does not mean that $x$ necessarily enters $T_2$ or at least $T_1$. Indeed, apriori it is not excluded
that all points $f(x),f^2(x),\dots$ are in vertical arcs of $D$ and none of them is in the horizontal one.}
\smallskip

Fix $N$.
To simplify the notation, denote $T_N$ and $A_N$ just by $T$ and $A$, respectively. 
From now on, we will consider $T_0$ as the underlying space to work with. 
Since $f$ is generically chaotic on $D$ and $T_0$ is an invariant regular closed subset by Proposition~\ref{P:conditions}, $f$ is generically chaotic on $T_0$ by Lemma~\ref{L:zuzenieGCH}. 
Now suppose, on the contrary, that $A$ is dense in a nonempty open subset $U$ of $T_0$. Since dendrites are locally connected, we may additionally assume that $U$ is connected.
By continuity and \ref{ENUM:Sens0}, $f^n(U)$ is connected and nondegenerate for every $n$.

By Lemma~\ref{L:doplnkySubdendr-Hranica2}, the escape-boundary $\Bd^*(T)=\{c_j\colon j\in J\}$ 
of (the proper subdendrite) $T$ in $T_0$ is at most countable.
Put
\[
	T^*=\bigcup\{\overline{C}\colon C\text{ is a component of } T_0\setminus T\}
	\subseteq\overline{T_0\setminus T}.
\]
For every $j\in J$ denote by $B_j$ the component of $T^*$ containing $c_j$
(that is, $B_j$ is the union of the closures of those components of $T_0\setminus T$ 
whose boundaries are $\{c_j\}$).
By Lemma~\ref{L:doplnkySubdendr-Hranica2}, every $B_j$ is a nondegenerate subdendrite of $T_0$
with $\Bd(B_j)=\{c_j\}$ 
and $d(B_i,B_j)>0$ for any distinct $i,j\in J$. 
By the choice of $T_n$'s, $p\notin B_j$ for any $j\in J$.

Since $A$ is dense in $U$, we have that $f^n(A)$ is dense in $f^n(U)$ for every $n\in \mathbb N_0$. 
By definition of~$A$, 
\begin{equation}\label{EQ:fnAcapT}
 f^n(A)\cap T=\emptyset
 \qquad\text{for every } n\in\NNN_0.
\end{equation}

Now fix $n\in \mathbb N_0$. We will show that 
\begin{equation}\label{EQ:fnUcapT}
f^n(U)\cap T \text{ contains at most one point and this point, if it exists, belongs to } \Bd\nolimits^*(T).
\end{equation}
First suppose that there are two distinct points $x,y\in f^n(U)\cap T$. Since both $f^n(U)$ and $T$ are connected, the whole arc $[x,y]\subseteq f^n(U)\cap T$. As $[x,y]$ is a regular closed set
by Proposition~\ref{P:conditions}, the set $f^n(A)$ is dense in $[x,y]\subseteq T$ and so $f^n(A)\cap T\ne\emptyset$, a contradiction with \eqref{EQ:fnAcapT}. We have thus shown that $f^n(U) \cap T$ has at most one element. 
Now suppose that $f^n(U) \cap T=\{y\}$.
The set $f^n(U)$ is nondegenerate, connected, contains the point $y\in T$ 
and $f^n(U)\setminus\{y\} \subseteq T_0\setminus T$.
Thus there is a component $B_{j_0}$ of $T^*$ such that
$f^n(U)\subseteq B_{j_0}$; clearly, 
$\{y\}$ coincides with $\Bd(B_{j_0})=\{c_{j_0}\}$. So $y$ belongs to the escape-boundary $\Bd^*(T)$. We have proved \eqref{EQ:fnUcapT}.
 
Fix any $n\in\NNN_0$. The set $f^n(U)$ is connected and nondegenerate. If
$f^n(U)\cap T\ne\emptyset$ then, by \eqref{EQ:fnUcapT}, $f^n(U)$ contains some $c_j$
and so $f^n(U)\subseteq B_j$. If $f^n(U)\cap T=\emptyset$ then $f^n(U)$ is a subset of some component of $T_0\setminus T$,
so again there is $j\in J$ with $f^n(U)\subseteq B_j$. That is,
\begin{equation}\label{EQ:fnU-in-Bj}
	\text{for every }n\in\NNN_0
	\text{ there is a unique } j\in J
	\text{ such that }
	f^n(U)\subseteq B_{j}.
\end{equation}

Since $f$ is generically chaotic on $T_0$ and $U$ is nonempty open in $T_0$, 
$\limsup_{n\rightarrow \infty} \diam f^n(U)>0$ by \ref{ENUM:Sens0}.
Since $U$ is connected, Lemma~\ref{L:finiteLength} yields the existence of
a minimal $n_0\in\mathbb N_0$ and a corresponding minimal $k\in\NNN$ such that
\[
f^{n_0}(U)\cap f^{n_0+k}(U)\neq \emptyset.
\]
Then, in view of \eqref{EQ:fnU-in-Bj}, the sets $f^{n_0}(U)$ and $f^{n_0+k}(U)$ 
lie in the same component of $T^*$.
It follows that for any $i\in \{0,1,\dots, k-1\}$ and for any $\ell\in \mathbb N$, 
the sets $f^{n_0+i}(U)$ and $f^{n_0+i+\ell k}(U)$ lie in the same component of
$T^*$. This implies that the orbit of $U$ intersects only finitely many components of $T^*$,
i.e.,
\begin{equation*}
  \Orb\nolimits_f(U)\subseteq \bigcup_{j\in J'} B_j
\end{equation*}
for some finite subset $J'$ of $J$.
Hence $\Orb_f(U)$ is disjoint from some $\eps$-neighbourhood $B_\eps(p)$ of $p$.

Since $\diam(T_m)\to 0$, we can choose $m$ with $T_m\subseteq B_{\eps/2}(p)$.
Then $\Orb_f(T_m)\subseteq T_m$ has positive distance from $\Orb_f(U)$. Since both $T_m$ and $U$
have nonempty interiors in $T_0$, this contradicts the fact that $f\colon T_0\to T_0$ is generically chaotic.
This finishes Step~2 of the proof.

\smallskip
\textit{Step~3.} We finish the proof by finding a contradiction with generic chaoticity of $f$ on $T_0$.
\smallskip

By Step~2, each of the sets $A_n\subseteq T_0$ ($n\in \NNN$) of points whose orbits do not intersect $T_n$
is nowhere dense in $T_0$.
Clearly, $\{A_n\}_{n=1}^{\infty}$ is an increasing sequence in the sense that $A_n\subseteq A_{n+1}$.
Consider the set 
\[
M=\bigcup_{n=1}^{\infty}A_n.
\]
Then $K=T_0\setminus M$ is the set of points from $T_0$ whose orbits intersect $T_n$ for every $n\in\NNN$. 
The trajectory of every such point actually converges to $p$, because all the sets $T_n$ are invariant, contain
$p$ and their diameters tend to zero. Consequently, for every $x,y\in K$ we have $d(f^n(x), f^n(y))\to 0$.
Since each $A_n$ is nowhere dense in $T_0$, the set $K$ is residual in $T_0$ and $K^2$ is residual in $T_0^2$.
This contradicts the fact that $f$ is generically chaotic on $T_0$.
\end{proof}

\begin{lemma} \label{L:subdend_at_least_delta}
Let $D$ be a completely regular nondegenerate dendrite with all points of finite order. Let $f\colon D\to D$ be generically chaotic. Then there exists $\delta>0$ such that the diameter of each invariant nondegenerate subdendrite is at least $\delta$.
\end{lemma}

\begin{proof}
Suppose, on the contrary, that there exists a sequence $\{D_n\}_{n=1}^\infty$ of invariant subdendrites with $\diam D_n\to 0$. By Lemma~\ref{L:fixed_pt_in_intersection} (note that our assumptions imply that the condition
\ref{ENUM:B1} is satisfied by the family of nondegenerate subdendrites), there exists a fixed point $p\in\bigcap_{n=1}^{\infty}D_n$. This contradicts Lemma~\ref{L:sub_cont_fix_pt-diam_delta}.
\end{proof}

The following lemmas show some connections between generic chaos and the orbits of subdendrites.

\begin{lemma}\label{L:orbits_with_fixed_pts-diameter_delta}
Let $D$ be a completely regular nondegenerate dendrite with all points of finite order. Let $f\colon D\to D$ be generically chaotic. Then there exists $\delta>0$ such that every nondegenerate subdendrite $E\subseteq D$ with $\Orb_f(E)$ containing a fixed point satisfies $\limsup_{n\rightarrow \infty} \diam f^n(E)>\delta$.
\end{lemma}

\begin{proof}
Suppose, on the contrary, that there exists a sequence of nondegenerate subdendrites $\{D_i\}_{i=1}^\infty$ such that for every $i$ there is some fixed point $p_i\in\Orb_f(D_i)$ and 
\[
\eps_i\to 0,
\quad\text{where}\quad
\eps_i=\limsup_{n\rightarrow \infty} \diam f^n(D_i).
\]
Here $\eps_i >0$ because $D_i$ has nonempty interior and $f$ is generically chaotic, see \ref{ENUM:Sens0}.
For every $i\in \mathbb N$ we can pick $n_i\in \mathbb N$ such that $f^n(D_i)$ contains $p_i$ and $\diam f^n(D_i)<2\eps_i$ whenever $n \geq n_i$. Consider the subdendrites
\[
E_i=\overline{\bigcup_{n=n_i}^{\infty} f^n(D_i)}.
\]
It can be easily concluded that $\diam E_i\leq 4\eps_i$. So the sets $E_i$ are invariant nondegenerate subdendrites with $\diam E_i\to 0$, which contradicts Lemma~\ref{L:subdend_at_least_delta}.
\end{proof}

\begin{lemma}\label{L:orbits_with_fD_intersects_D-diameter_delta}
Let $D$ be a completely regular nondegenerate dendrite with all points of finite order. Let $f\colon D\to D$ be generically chaotic. Then there exists $\delta>0$ such that 
\[
\limsup_{n\rightarrow \infty} \diam f^n(E)>\delta
\] 
whenever $E\subseteq D$ is a (nondegenerate) subdendrite satisfying $E\cap f(E)\neq \emptyset$ and having $\Orb_f(E)$ fixed point free.
\end{lemma}

\begin{proof}
Suppose on the contrary that there exists a sequence of nondegenerate subdendrites $\{D_i\}_{i=1}^{\infty}$ such that $D_i\cap f(D_i)\neq \emptyset$, $\Orb_f(D_i)$ is fixed point free and 
\[
\eps_i\to 0,
\quad\text{where} \quad
\eps_i=\limsup_{n\rightarrow \infty} \diam f^n(D_i).
\]
By \ref{ENUM:Sens0}, $\eps_i>0$.
Now fix $i\in \mathbb N$. 

\smallskip
\emph{Step~1.} We define a subdendrite $F_i$.
\smallskip

The set $F_i=\overline{\Orb_f(D_i)}$ is an invariant nondegenerate subdendrite of $D$ since $\Orb_f(D_i)$ is connected. 
From now on we will work with $F_i$ (with relative topology). 
Since $\Orb_f(D_i)$ is fixed point free, by Lemma~\ref{L:endS}
\[
  \emptyset \ne \Fix(f) \cap F_i \subseteq \End(F_i).
\]

\smallskip
\emph{Step~2.} We find an appropriate endpoint $a_i\in F_i\setminus\Orb_f(D_i)$ 
fixed by $f$, and points $x_j\in f^j(D_i)$ ($j\in\NNN_0$).
\smallskip

By Lemma~\ref{L:noneffluent_endpoint-every_point_admires} there exists a fixed endpoint $a_i\in F_i$ such that every non-fixed $y\in F_i$ admires~$a_i$. By the assumption, $a_i\notin \Orb_f(D_i)$ and every point from $\Orb_f(D_i)$
admires $a_i$. For every $j\in\NNN_0$ put
\[
  x_j = \proj(a_i,f^j(D_i)) \in f^j(D_i).
\]
Since $x_j\in f^j(D_i)\subseteq \Orb_f(D_i)$, it is not fixed by $f$.

\smallskip
\emph{Step~3.} We prove that $x_{j-1}\in f^j(D_i)$ for every $j\in \NNN$.
\smallskip

Suppose on the contrary that $x_{j-1}\notin f^j(D_i)$ for some $j\in\NNN$. 
As $x_{j-1}\in f^{j-1}(D_i)$, we have $f(x_{j-1})\in f^j(D_i)$, and $x_{j-1}$ as a non-fixed point of $F_i$ admires $a_i$. In other words $f(x_{j-1})\in (F_i)_{[a_i,x_{j-1})}$ (see \eqref{Eq:Dab}). 
Since $D_i\cap f(D_i)\ne\emptyset$, there is $y\in f^{j-1}(D_i)\cap f^{j}(D_i)$. Obviously $y\neq x_{j-1}$, as $x_{j-1}\notin f^j(D_i)$. Then by definition of $x_{j-1}$ we get $y\in (F_i)^{a_i}(x_{j-1})\setminus\{x_{j-1}\}$ (see \eqref{Eq:Max}). It follows that $x_{j-1}\in (f(x_{j-1}),y)\subseteq f^j(D_i)$, a contradiction. So indeed 
\[
x_{j-1}\in f^j(D_i).
\]

\smallskip
\emph{Step~4.} We prove that $x_j\in (a_i,x_{j-1})$ for every $j\in \NNN$.
\smallskip

By Step~3, $x_{j-1}\in f^j(D_i)$. 
However,  $x_j = \proj({a_i}, f^j(D_i))$ and so
$x_j\in [a_i,x_{j-1}]$ by Lemma~\ref{L:convexMetric2}\eqref{ENUM:convex-arc_xe}.
Since $x_j\in\Orb_f(D_i)$ and $a_i\not\in\Orb_f(D_i)$, we get $x_j\ne a_i$.
Further, $x_{j-1}$ admires $a_i$, i.e., $f(x_{j-1})\in (F_i)_{[a_i,x_{j-1})}$.
Now $f(x_{j-1})\in f^j(D_i)$ and the definition of $x_j$ give that $x_j\in (F_i)_{[a_i,x_{j-1})}$.
Thus $x_j\ne x_{j-1}$.

\smallskip
\emph{Step~5.} We prove that $\lim_{j\to\infty} x_j=a_i$.
\smallskip

By Step~4, all points $x_j$ are in $(a_i,x_0]$ with $x_{j+1}<_{a_i}x_j$.
Therefore the sequence $\{x_j\}_{j=1}^\infty$ monotonically converges to some $x\in [a_i,x_0)$.
We want to show that $x=a_i$. 
For contradiction let $x\in (a_i,x_0)$. 
Then $x$ is a cutpoint of $F_i=\overline{\Orb_f(D_i)}$ and so $x\in\Orb_f(D_i)$ by Lemma~\ref{L:endS}.
Hence $x$ admires $a_i$, i.e., $f(x)\in (F_i)_{[a_i,x)}$. 
The continuity of $f$ implies that for some index $K$, the point $x_K$, which is close enough to $x$, 
has its image $f(x_K)$ in $(F_i)_{[a_i,x)}$. As $f(x_K)\in f^{K+1}(D_i)$, 
we get $f^{K+1}(D_i)\cap (F_i)_{[a_i,x)}\neq \emptyset$ and hence $x_{K+1}\in (F_i)_{[a_i,x)}$. 
On the other hand, $x_{K+1}\in[a_i,x_0)$ by Step~4.
This gives that $x_{K+1}\in [a_i,x)$, which contradicts the definition of~$x$.

\smallskip
\emph{Step~6.} We construct a nondegenerate invariant subdendrite $E_i$ of $D$ with $\diam E_i\le 6\eps_i$.
\smallskip

Let $U_i$ be the $\eps_i$-neighbourhood of $a_i$ (in $F_i$). By Step~4, choose $N_1\in \mathbb N$ such that $x_n\in U_i$ whenever $n\geq N_1$. Hence $f^n(D_i)\cap U_i\neq \emptyset$ for every $n\geq N_1$. 
By the definition of $\eps_i$, there is  $N_2\in \mathbb N$ such that $\diam f^n(D_i)<2\eps_i$ for every $n\geq N_2$. Put $N=\max\{N_1,N_2\}$; then
\[
E_i=\overline{\bigcup_{n=N}^{\infty}f^n(D_i)}
\]
is a nondegenerate invariant subdendrite of $F_i$ (hence also of $D$) with diameter at most~$6\eps_i$.

\smallskip
\emph{Step~7.} We finish the proof.
\smallskip

Since fixed $i\in \NNN$ was arbitrary, we have a sequence $\{E_i\}_{i=1}^\infty$ of nondegenerate 
invariant subdendrites of $D$ with diameters converging to zero. 
By Lemma~\ref{L:subdend_at_least_delta} this contradicts generic chaoticity of~$f$. 
\end{proof}

\begin{lemma}\label{L:cycle}
Let $D$ be a nondegenerate dendrite. 
Let $f\colon D\to D$ be continuous and $E\subseteq D$ be a subdendrite with $\limsup_{n\rightarrow \infty}\diam f^n(E)>0$. Let $L_0,L_1,\dots,L_{r-1}$ be as in Lemma~\ref{L:finiteLength}.
If $f$ is generically chaotic, $r>1$ and $L_0,L_1,\dots, L_{r-1}$ have nonempty interiors, then 
$\bigcap_{i=0}^{r-1} \overline{L_i}$ is a singleton.
\end{lemma}

\begin{proof}
The map $g=f^r$ is generically chaotic. Take any distinct $i, j\in \{0,1,\dots, r-1\}$. Since $L_i$ and $L_j$ have nonempty interiors, there exists a Li-Yorke pair in $L_i\times L_j$. This implies that $d(L_i,L_j)=0$, hence $\overline{L_i}\cap \overline{L_j}\neq\emptyset$. 
So $L_i\cap L_j=\emptyset\neq\overline{L_i}\cap \overline{L_j}$ for any distinct $i,j$. 
By Lemma~\ref{L:finiteCap2}, $\bigcap_{i=0}^{r-1} \overline{L_i}$ is a singleton.
\end{proof}

\begin{lemma}\label{L:proper-connection}
Let $D$ be a completely regular nondegenerate dendrite. Let $f\colon D\to D$ be generically chaotic and $E\subseteq D$ be a nondegenerate subdendrite such that the following conditions are satisfied:
\begin{enumerate}
\item $\Orb_f(E)$ is connected;
\item there exists $k\in \mathbb N$ such that $E\cap f^k(E)\neq\emptyset$.
\end{enumerate}
Denote $K_j=\bigcup_{i=0}^{\infty}f^{j+ik}(E)$, $j\in\{0,1,\dots, k-1\}$. Then $\bigcap_{j=0}^{k-1}K_j\neq \emptyset$.  
\end{lemma}

\begin{proof} 
We assume that $k\ge 2$, otherwise the lemma is trivial.
The sets $K_j$ are nondegenerate and connected.
The map $g=f^k$ is generically chaotic and the sets $K_j$ are $g$-invariant. 
The subdendrites $\overline{K_j}$ are regular closed sets by Proposition~\ref{P:conditions}. So for any distinct $l,m\in\{0,1,\dots, k-1\}$ we can find a Li-Yorke pair in $\overline{K_l}\times\overline{K_m}$, hence $\overline{K_l}\cap\overline{K_m}\neq\emptyset$. 
By Lemma~\ref{L:finiteCap}, $\bigcap_{j=0}^{k-1}\overline{K_j}\neq\emptyset$.

For contradiction suppose that $\bigcap_{j=0}^{k-1}{K_j}=\emptyset$. 
By Lemma~\ref{L:finiteCap2}, $\bigcap_{j=0}^{k-1}\overline{K_j}$ is a singleton $\{p\}$. 
By Lemma~\ref{L:finiteLength} (with $n_0=0$), the sets $K_j$ are cyclically permuted by $f$ in the sense of 
\eqref{EQ:cycPerm}. Then also their closures are cyclically permuted by $f$
and since $p\in \bigcap_{j=0}^{k-1}\overline{K_j}$, then also $f(p)\in \bigcap_{j=0}^{k-1}\overline{K_j}=\{p\}$. Hence $p$ is fixed by $f$.

The point $p$ belongs to some $K_j$, otherwise $\Orb_f(E)=\bigcup_{j=0}^{k-1} K_j$ would not be connected. 
But then the point $p$, being fixed for $f$, belongs to every $K_j$ by \eqref{EQ:cycPerm}.
This contradicts the assumption that $\bigcap_{j=0}^{k-1}{K_j}=\emptyset$.
\end{proof}

Let $D$ be a nondegenerate dendrite and $f\colon D\to D$ be generically chaotic. Given a nondegenerate 
subdendrite $E \subseteq D$, in the following lemma we will work with the set
\begin{equation}\label{Eq:predef k}
  M_f(E) = 
  \{l\in \mathbb N \colon \text{there exists } n \geq 0 \text{ such that } f^{n}(E) \cap f^{n+l}(E)\neq\emptyset\}.
\end{equation}
If $D$ is completely regular, $E$ has nonempty interior 
and so $\limsup_{n\rightarrow\infty}\diam f^n(E)>0$ by \ref{ENUM:Sens0}.
Then $M_f(E)\ne\emptyset$ by Lemma~\ref{L:finiteLength}. 

\begin{lemma}\label{L:k_equals_1}
Let $D$ be a completely regular nondegenerate dendrite. Let $f\colon D\to D$ be generically chaotic. Let $E\subseteq D$ be a nondegenerate subdendrite and $k=\min M_f(E)$. 
Assume that the following conditions are satisfied:
\begin{enumerate}
\item\label{ENUM:OrdConn} $\Orb_f(E)$ is connected and periodic point free;
\item\label{ENUM:fkEcapE} $E\cap f^k(E)\neq\emptyset$.
\end{enumerate}
Then $k=1$. 
\end{lemma}

\begin{proof}
Due to \eqref{ENUM:OrdConn}, the set $\overline{\Orb_f(E)}$ is a subdendrite of $D$. Since it is invariant and, by Proposition~\ref{P:conditions}, also regular closed, the restriction of $f$ to this set is generically chaotic by
Lemma~\ref{L:zuzenieGCH}. Therefore, without loss of generality we may assume that
$$
D=\overline{\Orb\nolimits_{f}(E)},
$$
otherwise we just restrict our dynamical system to the subdendrite 
$\overline{\Orb_f(E)}$. 
Notice that $\Orb_f(E)$, being a connected dense set, contains all cutpoints of $D$.

Suppose that $k\ge 2$. 
To get a contradiction, we proceed in several steps. 

\smallskip
\textit{Step~1.} We find a subdendrite $\widetilde E$ that is an appropriate 
iterate of $E$, and define subdendrites $A_n,B_n$ ($n\in\NNN_0$). 
\smallskip

Since $f$ satisfies \ref{ENUM:Sens0}, 
the condition \eqref{ENUM:fkEcapE} and Lemma~\ref{L:finiteLength} (with $n_0=0$) imply that
\[
  \Orb\nolimits_f(E)=\bigcup_{j=0}^{k-1}K_j, 
  \qquad\text{where}\quad
  K_j=\bigcup_{i=0}^{\infty}f^{j+ik}(E)
  \quad(j\in\{0,1,\dots,k-1\}). 
\]
Fix two distinct $j,j' \in\{0,1,\dots,k-1\}$. By Lemma~\ref{L:proper-connection} we know that $K_j\cap K_{j'}\neq\emptyset$. So, there are nonnegative integers $i$, $i'$ such that 
\begin{equation}\label{Eq:i,j}
f^{j+ik}(E)\cap f^{j'+i'k}(E)\neq\emptyset.
\end{equation}
Choose $i$ and $i'$ satisfying~\eqref{Eq:i,j} which minimize the quantity $\Delta=|(j'+i'k)-(j+ik)|$.
Clearly $\Delta>0$, and so $\Delta \geq k$ due to the definition of $k$. We have in fact $\Delta \geq k+1$ because $0<|j'-j|<k$. Without loss of generality we may assume that  $j'+i'k > j+ik$, so 
$$
\Delta = (j'+i'k)-(j+ik) \geq k+1.
$$ 
Consider the (nondegenerate) dendrite 
\[
\widetilde E=f^{j+ik}(E).
\]
Put $m= \Delta -k$. Then $m \geq 1$ and $m$ is not a multiple of $k$. 
In the sequel we will denote $g=f^k$ and
\[
  A_n=f^{nk}(\widetilde E)=g^n(\widetilde E)
  \quad\text{and}\quad
  B_n=f^{nk+m}(\widetilde E)=g^n(f^m(\widetilde E))
  \qquad\text{for}\quad n\in\NNN_0.
\]
The sets $A_n,B_n$ are nondegenerate subdendrites.

\smallskip
\textit{Step~2.} We study $g$-iterates of the sets $A_n$ and $B_n$ from the 
point of view of their intersections.
\smallskip

We claim that
\begin{equation}\label{Eq:pretnu}
A_0\cap B_1 = \widetilde E\cap f^{k+m}(\widetilde E)\neq\emptyset
\quad \text{and} \quad
A_1\cap B_1 = f^{k}(\widetilde{E})\cap f^{k+m}(\widetilde{E})
=\emptyset.
\end{equation}
The first part is (\ref{Eq:i,j}). Further,
$f^{k}(\widetilde{E})\cap f^{k+m}(\widetilde{E}) = f^{j+(i+1)k}({E}) 
\cap f^{j'+i'k}({E})$ and $(j'+i'k)- (j+(i+1)k) = m < k+m = \Delta$ and so 
$A_1\cap B_1=\emptyset$ by minimality of $\Delta$. 

By applying $g=f^k$ to the 
nonempty intersection in~\eqref{Eq:pretnu} and, 
respectively, by minimality of $\Delta$, we get
\begin{equation*}
A_1\cap B_2=f^{k}(\widetilde{E})\cap f^{2k+m}(\widetilde{E})\neq\emptyset
\quad \text{and} \quad
A_2\cap B_2=f^{2k}(\widetilde{E})\cap f^{2k+m}(\widetilde{E})=\emptyset.
\end{equation*}
By induction, for all $n\in \mathbb N$,
\begin{equation}\label{Eq:pretnun}
A_{n-1}\cap B_n  \neq\emptyset
\quad \text{and} \quad
A_n\cap B_n=\emptyset.
\end{equation}
Further, by \eqref{ENUM:fkEcapE},
\begin{equation*}
A_n\cap A_{n+1} = f^{nk}(\widetilde{E})\cap f^{(n+1)k}(\widetilde{E})\neq\emptyset
\end{equation*}
and
\begin{equation*}
B_n\cap B_{n+1} = f^{nk+m}(\widetilde{E})\cap f^{(n+1)k+m}(\widetilde{E})\neq\emptyset
\end{equation*}

\smallskip
\textit{Step~3.} We partially describe the ``position'' of the sets $A_n,B_n$ in $D$. 
\smallskip

Using~\eqref{Eq:pretnun} we get that for all $n\in N$
\begin{equation*}
A_n\subseteq C_n\quad \text{for some (unique) component } C_n \text{ of } 
D\setminus B_n.
\end{equation*}
Choose $n\in\NNN$ and apply Lemma~\ref{L:doplnkySubdendr} to the sets $A_n,A_{n+1}$ and $B_n,B_{n+1}$.
Denoting the boundary of $C_i$ by $\{x_i\}$ ($i=n,n+1$), we have
\begin{equation}\label{Eq:AnCn}
	A_n\subseteq C_n,\quad
	C_{n}\supsetneq C_{n+1},\quad
	x_{n+1}\in A_{n}\cap B_{n+1},\quad
	x_n\notin \overline{C_{n+1}}, \quad
	\{x_{n+1}\} = \Bd(\overline{C_{n+1}}).
\end{equation}
It also follows that $x_{n+1}\notin C_{n+1}$ since $C_{n+1}$ is open. Further, 
\begin{equation}\label{EQ:xn-ne-xn1}
	x_n\ne x_{n+1}
	\qquad\text{for every}\  n\in\NNN.
\end{equation}

\smallskip
\textit{Step~4.} We study the intersection $C=\bigcap_{n=1}^\infty C_n$. 
\smallskip

Since $\widetilde{E}$ is a subdendrite with nonempty interior and $g=f^k$ is generically chaotic,  
\[
  \delta
  =\limsup_{n\rightarrow\infty} \diam A_n 
  =\limsup_{n\rightarrow\infty} \diam g^n(\widetilde E)
  >0.
\] 
Since $C_1\supsetneq C_2\supsetneq\dots$ is a nested sequence
and $C_n\supseteq A_n$, it follows that for all $n$ we have $\diam C_n \geq \delta$. 
Notice that
$$
C=\bigcap_{n=1}^\infty C_n
=\bigcap_{n=1}^\infty \overline{C_n}.
$$
Indeed, this follows from the fact that, for every $n$, the boundary of $C_n$ is $\{x_n\}$ and $x_n\notin C_{n+1}$.
Thus $C$ is a nondegenerate dendrite.

\smallskip
\textit{Step~5.} By considering the limit of the sequence $\{x_n\}_{n=1}^\infty$ we get a contradiction. 
\smallskip

By \eqref{Eq:AnCn}, for $n\ge 2$ we have $\{x_n\}=\Bd(\overline{C_n})$
and $x_1\notin \overline{C_n}$.
So, by Lemma~\ref{L:convexMetric2}\eqref{ENUM:convex-singletonBoundary},
\begin{equation*}
	x_n=\proj(x_1, \overline{C_n}).
\end{equation*}
Put 
\[
  x = \proj(x_1, C).
\]
Let $n\ge 2$. Then $x\in C\subseteq \overline{C_n}$. On the other hand, $x_1\notin \overline{C_n}$ by \eqref{Eq:AnCn}.
Hence, by Lemma~\ref{L:convexMetric2}\eqref{ENUM:convex-arc_xe}, 
\begin{equation}\label{EQ:xn-x1}
	x_n\in [x_1,x].
\end{equation}
By \eqref{EQ:xn-x1}, \eqref{EQ:xn-ne-xn1} and the fact that $\{\overline{C_n}\}_{n=1}^\infty$ is a nested sequence, 
we get that
\begin{equation}\label{EQ:xn-ordered}
 \{x_n\}_{n=1}^{\infty} \text{ is a linearly ordered set with }
 x_{n+1}<_{x}x_n.  
\end{equation}

We claim that (see \eqref{Eq:Max})
\begin{equation}\label{EQ:C-and_leq}
C = \{y\in D\colon x\leq_{x_1} y\}=D^{x_1}(x).
\end{equation}
Indeed, by the choice of $x$ we have $C\subseteq D^{x_1}(x)$. 
Now let $y\in D^{x_1}(x)$. To show that $y\in C$ we fix $n$ and show that $y\in {C_n}$. 
Put $A=(x_n,y]$. This is a connected set that intersects $C_n$ but does not contain the point $x_n$ which
is the unique boundary point of $C_n$.
Thus $A\subseteq C_n$ by Lemma~\ref{L:boundaryBumping} and so $y\in C_n$.

By Lemma~\ref{L:convexMetric2}\eqref{ENUM:convex-arc_xe},
\begin{equation*}
[x_n,x]\subseteq \overline{C_n} \subseteq \{y\in D\colon x_n\leq_{x_1} y\}=D^{x_1}(x_n)
\end{equation*}
for every $n\ge 2$. This together with \eqref{EQ:xn-x1}, \eqref{EQ:xn-ordered} and the definitions of $C$ 
and $x$ yield
\begin{equation}\label{EQ:lim-xn}
\lim_{n\rightarrow\infty}x_n=x.
\end{equation}

Since $C$ is nondegenerate, $x$ is a cutpoint of $D$ by Lemma~\ref{L:convexMetric2}\eqref{ENUM:convex-cutpoint},
hence $x\in\Orb_f(E)$. 
As $\Orb_f(E)$ is periodic point free, $x$ is not a fixed point for $g$. 
Further, by \eqref{Eq:AnCn}, $x_n\in A_{n-1}$; hence 
$g(x_n)\in g(A_{n-1})=A_n\subseteq C_{n}\subseteq \overline{C_n}$. 
Then, since $\{\overline{C_n}\}_{n=1}^\infty$ is a nested sequence, \eqref{EQ:lim-xn} and \eqref{EQ:C-and_leq}
give
\[
  g(x)=\lim_{n\to\infty} g(x_n)\in \bigcap_{n=1}^\infty \overline{C_n} = C = D^{x_1}(x).
\]
As $g(x)\neq x$, we have $g(x)\in D^{x_1}(x)\setminus\{x\}$. 
By continuity, $g(x_n)\in D^{x_1}(x)\setminus\{x\}\subseteq C$ for sufficiently large $n$. 
On the other hand, by \eqref{Eq:AnCn}, for every $n\ge 2$ we have $x_n\in B_n$ 
and so $g(x_n)\in B_{n+1}$. It follows that $g(x_n)\notin C_{n+1}$,
whence $g(x_n)\notin C$, a contradiction.
\end{proof}

\section{Proof of Theorem~\ref{T:dendrite_gch=gech}: (\ref{ENUM:T2}) implies (\ref{ENUM:T1}) }
\label{S:thm21}

We assume that $D$ is a completely regular (nondegenerate) dendrite with all points of finite order 
and $f\colon D\to D$ 
is a generically chaotic map. We are going to prove that $f$ is generically $\eps$-chaotic for some $\eps>0$. 
By Proposition~\ref{P:B1-B2} and the assumptions on $D$, 
it is sufficient to show that \ref{ENUM:B1} and \ref{ENUM:B2} 
are satisfied by the family of all nondegenerate subdendrites of $D$.
Since $f$ is generically chaotic, the condition \ref{ENUM:B1} is satisfied trivially (also \ref{ENUM:Sens0}
is satisfied trivially).

To prove \ref{ENUM:B2} suppose, on the contrary, that there exists a sequence 
$\{\tilde{D_i}\}_{i=1}^\infty $ of nondegenerate subdendrites of $D$ such that
\begin{equation}\label{E:diameters_tend_to_infty}
\lim_{i\rightarrow\infty}\limsup_{n\rightarrow \infty}\diam f^n(\tilde{D_i})=0,
\end{equation}
where, by \ref{ENUM:Sens0}, for every $i$ we have 
\begin{equation}\label{EQ:limsup-positive}
\limsup_{n\rightarrow \infty}\diam f^n(\tilde{D_i})>0.
\end{equation}
To get a contradiction, we proceed in several steps.

\smallskip
\textit{Step~1.} For $i\in\NNN$ we define positive integers $m_i$ and replace $\tilde{D_i}$ by $D_i$. 
\smallskip

The set $M_f(\tilde{D_i})$ from \eqref{Eq:predef k} is nonempty due to \ref{ENUM:Sens0};
put $m_i=\min M_f(\tilde{D_i})$. Fix $n_i\ge 0$ such that $f^{n_i}(\tilde{D_i}) \cap f^{n_i+m_i}(\tilde{D_i}) \ne \emptyset$ and put
\[ 
    D_i=f^{n_i}(\tilde{D_i}) \qquad(i\in\mathbb N).
\] 
Then \eqref{E:diameters_tend_to_infty} and \eqref{EQ:limsup-positive} still hold for $D_i$ instead of $\tilde{D_i}$.
Moreover, 
\begin{equation}\label{EQ:Di-fmiDi}
	D_i\cap f^{m_i}(D_i)\neq\emptyset
	\qquad (i\in\NNN).
\end{equation}
Clearly, $m_i=\min M_f({D_i})$.

\smallskip
\textit{Step~2.} For $i\in\NNN$ we define positive integers $r_i$, connected sets $K^k_i$ and $L^j_i$ and points $p_i$. 
\smallskip

By \eqref{EQ:Di-fmiDi}, for $i\in\NNN$ the sets 
\[
K^k_i=\bigcup_{n=0}^\infty f^{k+n m_i}(D_i)
\qquad(k\in \{0,1,\dots,m_i-1\})
\]
are connected and, by \eqref{EQ:limsup-positive}, they are nondegenerate. 
By Lemma~\ref{L:finiteLength}, for any $i\in \mathbb N$ there exist $r_i\in\NNN$ (a divisor of $m_i$) and connected sets $L^0_i,L^1_i,\dots,L^{r_i-1}_i$ such that
\[
\Orb\nolimits_f(D_i)=L^0_i\sqcup L^1_i\sqcup \dots \sqcup L^{r_i-1}_i,
\]
$f(L^0_i)= L^1_i$, $f(L^1_i)= L^2_i,\dots$, $f(L^{r_i-2}_i)= L^{r_i-1}_i$ and $f(L^{r_i-1}_i)\subseteq L^0_i$. 
Here, 
\[
L^j_i=\bigcup_{k=0}^{{m_i}/{r_i}-1} K^{j+k r_i}_i
\qquad (j\in\{0,1,\dots, r_i-1\}).
\]
The closures
$\overline{L_i^j}$ are $f^{r_i}$-invariant
nondegenerate subdendrites.
By Lemma~\ref{L:cycle}, if $r_i>1$ then there exists a unique $p_i$ such that
\[
  \bigcap_{j=0}^{r_i-1}\overline{L^j_i}=\{p_i\}.
\]
Since the sets $L^j_i$ ($j\in\{0,\dots,r_i-1\}$) are pairwise disjoint,
they belong to different components of $D\setminus \{p_i\}$. Thus
\begin{equation}\label{EQ:order-pi}
	\ord(p_i,D)\ge r_i.
\end{equation}

\smallskip
\textit{Step~3.} To finish the proof, i.e., to get a contradiction, we distinguish two cases depending on whether the sequence $\{r_i\}_{i=1}^\infty$ is bounded or unbounded.
\smallskip

\smallskip
\textit{Case I.} Assume that the sequence $\{r_i\}_{i=1}^\infty$ is bounded.
\smallskip

By passing to a subsequence if necessary, we may assume that the sequence is constant, $r_i=r$ for every $i$.
Then $r$ divides every $m_i$. The map $g=f^r$  is generically chaotic. 
Put $L_i=L^0_i$ for every $i$. Recall that each $L_i$ is nondegenerate, $g$-invariant and connected,
with 
\begin{equation}\label{E:piece_is_union_of_prepieces}
L_i=\Orb\nolimits_g(D_i)=\bigcup_{k=0}^{{m_i}/{r}-1}K^{rk}_i.
\end{equation}
From \eqref{E:diameters_tend_to_infty} and generic chaoticity of $g$ we get
\begin{equation}\label{E:diameters_tend_to_infty-for_g}
\lim_{i\rightarrow\infty}\eps_i=0,
\qquad\text{where}\quad
\eps_i=\limsup_{n\rightarrow \infty}\diam g^n(D_i)>0.
\end{equation}
By (\ref{E:diameters_tend_to_infty-for_g})  and Lemma~\ref{L:orbits_with_fixed_pts-diameter_delta},
$\Orb_g(D_i)$ contains a fixed point of $g$ only for finitely many indices $i$. We may assume that 
\begin{equation}\label{Li-fpfree}
  \Orb\nolimits_g(D_i) \cap \Fix(g) =\emptyset \qquad\text{for every}\quad i\in \mathbb N. 
\end{equation}

If $m_i=r$ for infinitely many $i$, then $D_i\cap g(D_i)\neq\emptyset$ for these indices, which together with (\ref{E:diameters_tend_to_infty-for_g}) contradicts Lemma~\ref{L:orbits_with_fD_intersects_D-diameter_delta}. Therefore, we may assume that 
\begin{equation}\label{EQ:mi>r}
   m_i>r \qquad\text{for all }i\in\NNN.
\end{equation}
By definition of $m_i$,
\begin{equation}\label{Eq:gn gn+1}
g^{n}(D_i)\cap g^{n+1}(D_i)=\emptyset \qquad\text{ for every}\quad n\in \mathbb N_0.
\end{equation}

Fix any $i\in \mathbb N$. 
Each $\overline{L_i}$ is a nondegenerate $g$-invariant dendrite. 
It is regular closed by Proposition~\ref{P:conditions}.
Moreover, $\overline{L_i}$ shares the properties of $D$; it is completely regular, with all points of finite order.
From now on till the end of Case~I we will work with this dendrite $\overline{L_i}$ and
we will write just $g$ rather than $g|_{\overline{L_i}}$.
Note that by Lemma~\ref{L:zuzenieGCH}, this restriction is generically chaotic. 

By \eqref{EQ:mi>r}, in (\ref{E:piece_is_union_of_prepieces}) we have the union of at least two sets. 
Recall that $\Orb_g(D_i)$ in \eqref{E:piece_is_union_of_prepieces} is connected and by \eqref{EQ:Di-fmiDi}
we have $D_i\cap g^{m_i/r}(D_i)= D_i\cap f^{m_i}(D_i)\ne\emptyset$.
Notice also that $K_i^{rk} = \bigcup_{n=0}^\infty f^{rk+nm_i}(D_i)= \bigcup_{n=0}^\infty g^{k+nm_i/r}(D_i)$,
$k=0,\dots,m_i/r-1$. Thus we may apply Lemma~\ref{L:proper-connection} to get that
\[
	\bigcap_{k=0}^{m_i/r-1}K^{rk}_i\neq \emptyset. 
\]

By \eqref{Li-fpfree}, fixed points of $g$ exists only in $\End(\overline{L_i})\setminus \Orb_g(D_i)$. Therefore Lemma~\ref{L:noneffluent_endpoint-every_point_admires} ensures the existence of a point 
\begin{equation}\label{EQ:ai-not-in-orbDi}
a_i\in\End(\overline{L_i})\setminus \Orb\nolimits_g(D_i)
\end{equation}
which is fixed for $g$ and such that
\begin{equation}\label{EQ:x-admires-ai}
	\text{every point from } \Orb_g(D_i) \text{ admires } a_i.
\end{equation}
By
\eqref{E:piece_is_union_of_prepieces}, $a_i\in \overline{K^{rk}_i}$ for some $k\in \{0,1,\dots,m_i/r-1\}$. 
As $g(a_i)=a_i$ and the sets $K_i^{rk}$ in \eqref{E:piece_is_union_of_prepieces} are cyclically permuted by $g$,
we have 
\begin{equation*}
	a_i\in \bigcap_{k=0}^{m_i/r-1} \overline{K^{rk}_i}.
\end{equation*}
Denote $a=a_i$. Distinguish two subcases.

\smallskip
\textit{Subcase I(1).} Assume that there exists an open (in $\overline{L_i}$) neighbourhood $U$ of $a$  with a singleton boundary $\{u\}$  such that
\begin{equation}\label{Eq:case a}
\text{for any }n\in \mathbb N_0, \text{ if }g^n(D_i)\cap U\neq \emptyset \text{ then }g^n(D_i)\cap(a,u)\neq\emptyset.
\end{equation}
\smallskip

As $a\in \overline{\Orb_g(D_i)}$, there exists $n_0$ with $g^{n_0}(D_i)\cap U\neq \emptyset$. By \eqref{Eq:case a}, $g^{n_0}(D_i)\cap(a,u)\neq\emptyset$. Let 
\[
  x_0=\proj(a, g^{n_0}(D_i)).
\]
By \eqref{EQ:ai-not-in-orbDi}, $x_0\ne a$. 
Choose $y\in g^{n_0}(D_i)\cap(a,u)$ and use Lemma~\ref{L:convexMetric2}\eqref{ENUM:convex-arc_xe} to get
$x_0\in (a,y]\subseteq (a,u)$.
By \eqref{EQ:x-admires-ai} and \eqref{EQ:ai-not-in-orbDi}, $g(x_0)\in (\overline{L_i})_{(a,x_0)}$. 
Thus
\[
  x_0\in (a,u)
  \qquad\text{and}\qquad
  g(x_0)\in (\overline{L_i})_{(a,x_0)}.
\]

Then, since $g(x_0)\in g^{n_0+1}(D_i)$, we have that
$x_1=\proj(a, g^{n_0+1}(D_i))\in (\overline{L_i})_{(a,x_0)}$.
Similarly as above we can show that also $x_1\in (a,u)$.
It follows that $x_1\in (a,x_0)$.  Therefore, by \eqref{Eq:gn gn+1}, $g^{n_0+1}(D_i)\subseteq (\overline{L_i})_{(a,x_0)}$.

By induction, denoting $x_n=\proj(a,g^{n_0+n}(D_i))$, we get that the sequence $\{x_n\}_{n=0}^{\infty}$ lies in $(a,u)$ and is such that, for every $n\in\NNN_0$,
\[
  a \ <_a \  x_{n+1} \ <_a  \ x_{n}
  \qquad(n\in\NNN),
\] 
and 
\begin{equation}\label{EQ:gDi}
	  g^{n_0+n+1}(D_i)\subseteq (\overline{L_i})_{(a,x_n)}.
\end{equation}

We claim that $x_n\to a$. Indeed, by \eqref{E:piece_is_union_of_prepieces} and \eqref{EQ:ai-not-in-orbDi}, 
$\liminf_{n\to\infty} d(a_i, g^{n_0+n}(D_i))=0$, hence $\liminf_{n\to\infty} d(a_i, x_n)=0$; now monotonicity yields that $x_n\to a$.
Since $x_n\rightarrow a$,  $\diam (\overline{L_i})_{(a,x_n)}\to 0$ and 
\eqref{EQ:gDi} yields that $\diam g^n(D_i)\to 0$, which contradicts generic chaoticity of $g$.

\smallskip
\textit{Subcase I(2).} Assume that for every open (in $\overline{L_i}$) neighbourhood $U$ of $a$  with a singleton boundary~$\{u\}$  
\begin{equation*}
\text{there is }n_U\in \mathbb N_0
\text{ such that }
 g^{n_U}(D_i)\cap U\neq\emptyset
 \text{ and }
 g^{n_U}(D_i)\cap(a,u)=\emptyset.
\end{equation*}

If such a neighbourhood $U$ is connected, $g^{n_U}(D_i)\subseteq U$.
(Otherwise, by Lemma~\ref{L:boundaryBumping}, $g^{n_U}(D_i)$ contains~$u$.
Choose $x\in g^{n_U}(D_i)\cap U$. By Lemma~\ref{L:arcs-and-one-point-boundary},
$ g^{n_U}(D_i)\cap [a,u) \supseteq [x,u)\cap [a,u)\ne\emptyset$, a contradiction.)
Since $a$ is an endpoint of $\overline{L_i}$ it has a basis of connected neighbourhoods with singleton boundaries. 
Consider a sequence of such neighbourhoods $U_k$ ($k\in\NNN$) with corresponding singleton boundaries $\{u_k\}$ such that $u_k\to a$ and $u_{k+1}<_a u_k$ for any $k\in \mathbb N$. To simplify the notation let $n_k=n_{U_k}$. Thus
\begin{equation*}
\text{for every } k\in\NNN
\text{ there is } n_k\in\NNN_0
\text{ such that}
\quad
g^{n_k}(D_i)\subseteq U_k
\quad\text{and}\quad
g^{n_k}(D_i)\cap (a,u_k)=\emptyset.
\end{equation*}
Since $a\in\overline{\Orb_g(D_i)} \setminus{\Orb_g(D_i)}$, we may assume that $n_{k+1}>n_k\geq k$. 

We claim that $\Orb_g(D_i)$ does not contain any periodic point of $g$. Assume the opposite is true, i.e., $p\in \Orb_g(D_i)$ is a periodic point. Thus $p\in g^{N}(D_i)$ for some $N$. Then $g^n(D_i)\cap \Orb_g(p)\neq\emptyset$ for all $n\geq N$. Also for some  $m>N$ the neighbourhood $U_m$ is small enough to be disjoint with the orbit of $p$, i.e., $U_m\cap \Orb_g(p)=\emptyset$. But also $g^{n_m}(D_i)\subset U_m$ and, 
since $n_m\ge m > N$,  $g^{n_m}(D_i)\cap \Orb_g(p)\neq\emptyset$, which is a contradiction. Hence $L_i=\Orb_g(D_i)$ is connected and periodic point free. Further, $D_i\cap g^{m_i/r}(D_i)\neq \emptyset$ by \eqref{EQ:Di-fmiDi}.
Since $m_i=\min M_f(D_i)$ we have $m_i/r=\min M_g(D_i)$ (see \eqref{Eq:predef k}). Then Lemma~\ref{L:k_equals_1} implies $m_i/r=1$, a contradiction with \eqref{EQ:mi>r}.

\smallskip
\textit{Case II.} Assume that the sequence $\{r_i\}_{i=1}^\infty$ is unbounded.
\smallskip

By passing to a subsequence we may assume $1<r_i<r_{i+1}$ for $i\in\NNN$.
Fix $i$ and consider the map $g=f^{r_i r_1}$. Recall that the sets $L_i^j$ ($j\in \{0,1,\dots,r_i-1\}$) and $L_1^{j'}$ ($j'\in \{0,1,\dots,r_1-1\}$) are $g$-invariant and
\[
\bigcap_{j=0}^{r_i-1}\overline{L_i^j}=\{p_i\}
\quad\text{and}\quad
\bigcap_{j'=0}^{r_1-1}\overline{L_1^{j'}}=\{p_1\}.
\]
Since $g$ is generically chaotic on $D$, for any $j,j'$ the sets $L_i^j$ and $L_1^{j'}$ have zero distance, as they have non-empty interiors. 

We claim that $p_i=p_1$. Suppose not. Since $r_i>1$, there exists $J\in \{0,1,\dots,r_i-1\}$ such that $L^J_i$ lies in a component of $D\setminus \{p_i\}$ which does not contain $p_1$. Similarly, there exists  $J'\in \{0,1,\dots,r_1-1\}$ such that $L_1^{J'}$ lies in  a component of $D\setminus \{p_1\}$ which does not contain $p_i$. This gives $d(L_i^J,L_1^{J'})\geq d(p_i,p_1)>0$, a contradiction.

We have proved that $p_i=p_1$ for any $i$. 
Since $r_i\to\infty$, \eqref{EQ:order-pi} gives that $\ord(p_1,D)=\infty$, 
which contradicts the assumption that all points of $D$ are of finite order.

\section{Exact maps on dendrites}\label{S:exact-maps}
The purpose of this section is to prove the following result,
which will be used in Section~\ref{S:thm12}.

\begin{proposition}\label{P:exact}
	Let $D$ be a dendrite and $A\subseteq D$ be either a singleton 
	or a nowhere dense arc. 
	Then there is an exact map $f\colon D\to D$ such that
	$$
		f(x)=x \quad\text{for every}\quad x\in A.
	$$
\end{proposition}

We start by recalling some results from
\cite{Spi2014} which will be used in the proof of this proposition.

Let $D$ be a nondegenerate dendrite. We say that a family $\CCc$ of
nondegenerate subdendrites of $D$ is \emph{dense} if $D\in\CCc$ and every nonempty
open set in $D$ contains some $C\in\CCc$.
The system of all nondegenerate closed subintervals of $I=[0,1]$ is denoted by
$\CCc_I$; we assume that $I$ is equipped with the Euclidean metric
$d_I$. 

\begin{definition}\label{D:length-exp} 
	Let $(D,d)$, $(D',d')$ be nondegenerate dendrites
	and $\CCc,\CCc'$ be dense families of nondegenerate subdendrites of $D,D'$, respectively.
	Let $\varrho>1$.
	We say that a continuous map $f\colon D\to D'$ is 
	\emph{$\varrho$-length expanding (with respect to $\CCc,\CCc'$)}
	if for every $C\in\CCc$ we have $f(C)\in\CCc'$ and
	\begin{equation*}
		f(C)=D'
		\qquad\text{or}\qquad
		\HHh^1_{d'}(f(C)) \ge \varrho\cdot \HHh^1_d(C).
	\end{equation*}
\end{definition}

The following is a part of \cite[Theorem~C]{Spi2014}.

\begin{proposition}\label{P:lel-maps}
	Let $\varrho>1$, $D$ be a nondegenerate dendrite and $a\in D$.
	Then there is a convex metric $d_{D,a}$ on $D$ compatible with
	the topology of $D$, and 
	continuous surjections $\varphi_{D,a}\colon I\to D$, 
	$\psi_{D,a}\colon D\to I$ such that the following are true:
	\begin{enumerate}
	\item $\HHh^1_{d_{D,a}}(D)=1$;
	\item $\varphi_{D,a}(0)=\varphi_{D,a}(1)=a$ and $\psi_{D,a}(a)=0$;
	\item the family $\CCc_{D,a}=\varphi_{D,a}(\CCc_I)$ is a dense family of nondegenerate subdendrites of $D$;
	\item $\varphi_{D,a}$ and $\psi_{D,a}$ are $\varrho$-length expanding 
			(with respect to $\CCc_I,\CCc_{D,a}$ and  $\CCc_{D,a},\CCc_I$, respectively).	
	\end{enumerate}
\end{proposition}


\begin{proof}[Proof of Proposition~\ref{P:exact}]
We may assume that $D$ is nondegenerate.
If $A$ is a singleton, the assertion follows from \cite[Corollary~E]{Spi2014}.
(Though not stated explicitly, an exact map constructed in the mentioned corollary
can have one or two prescribed fixed points, since the proof of it is based on
\cite[Theorem~D]{Spi2014}.)
Assume now that $A$ is an arc and fix $0<q<1<\varrho$. 
The proof is divided into eight steps.

\smallskip
\textit{Step~1.} We define subdendrites $D_k$ ($k\ge 0$). 
\smallskip

Put 
$\tilde D=\bigcup\{\overline{C}\colon C\text{ is a component of } D\setminus A\}$.
Since $A$ is nondegenerate and nowhere dense,
the set $\tilde D$ has $\aleph_0$ components; denote them by $D_k$ ($k\in \NNN$).
By Lemma~\ref{L:doplnkySubdendr-Hranica2}, 
every $D_k$ is a nondegenerate dendrite such that $\Bd(D_k)=D_k\cap A$ is a singleton $\{x_k\}$;
moreover, $x_k\ne x_{h}$ for every $k\ne h$.
In the sequel we sometimes write $D_0$ instead of $A$.
We believe that it will be convenient to call the subdendrite $D_k$ ($k\ge 1$) a \emph{bush rooted at $x_k$}.
Then $D$ is the union of $D_0=A$ (which is not a bush) and all the bushes.
	
\smallskip
\textit{Step~2.} We define a convex metric $d$ on $D$ compatible with the topology of $D$,
and surjective maps $\varphi_k\colon I\to D_k$ and $\psi_k\colon D_k\to I$ ($k\ge 0$) such that,
for $k\ge 1$, $\varphi_k$ and $\psi_k$ are $\varrho$-length expanding (with respect to $\CCc_I,\CCc_k$ and
$\CCc_k,\CCc_I$, respectively).
\smallskip

For every $k\in\NNN$, consider the metric
$d_k=d_{D_k,x_k}$ and maps $\varphi_k=\varphi_{D_k,x_k}$, $\psi_k=\psi_{D_k,x_k}$ 
obtained from Proposition~\ref{P:lel-maps}
applied to the dendrite $D_k$ and the point $x_k$. 
Then
\begin{equation}\label{EQ:phipsi0}
   \varphi_k(0)=\varphi_k(1)=x_k
   \qquad\text{and}\qquad
   \psi_k(x_k)=0,
\end{equation}
and
\begin{equation}\label{EQ:Ck}
	\CCc_k=\varphi_k(\CCc_I) \quad \text{is a dense family of nondegenerate subdendrites of}\  D_k.
\end{equation}
Fix a homeomorphism $\varphi_0\colon I\to A=D_0$ and denote $\psi_0=\varphi_0^{-1}$. Put $\CCc_0=\varphi_0(\CCc_I)$. Let $d_0$ be a convex metric on $A$ given by
$d_0(x,x')=d_I(\psi_0(x),\psi_0(x'))$; thus $\HHh^1_{d_0}(A)=1$.

For every $k\ge 0$ put $\lambda_k=(1-q)q^k$. Define a metric $d$ on $D$ by
\begin{equation*}
	d(x,x') = 
	\begin{cases}
		\lambda_k d_k(x,x')
		&\text{if } x,x'\in D_k \text{ for } k\in\NNN_0;
	\\
		\lambda_0 d_0(x,x_k) + \lambda_k d_k(x_k,x')
		&\text{if } x\in A, x'\in D_k \text{ for } k\in\NNN;
	\\
		\lambda_k d_k(x,x_k) + \lambda_0 d_0(x_k,x')
		&\text{if } x\in D_k \text{ for } k\in\NNN, x'\in A;
	\\
		\lambda_k d_k(x,x_k) + \lambda_0 d_0(x_k,x_h) + \lambda_h d_h(x_h,x')
		&\text{if } x\in D_k, x'\in D_h \text{ for } k\ne h \text{ in }\NNN.
	\end{cases}
\end{equation*}
Since dendrites are uniquely arcwise connected,
it is an easy exercise to show that $d$ is a convex metric compatible with
the topology of $D$ and that
\begin{equation*}
	\diam\nolimits_d D_k \le \HHh^1_d(D_k)=\lambda_k
	\qquad\text{for every } k\ge 0.
\end{equation*}
Hence $\diam\nolimits_d D \le\HHh^1_d(D)=\sum_{k=0}^\infty \lambda_k = 1$.
Notice that for $k,h\in\NNN_0$ we have 
\[
k< h 
\iff
\HHh^1_d(D_k) > \HHh^1_d(D_{h})
\]
and $\HHh^1_d(D_k)\searrow 0$ as $k\to\infty$. In particular, if $D_k$ and $D_h$ are two bushes
with $k< h$, it is reasonable to say that $D_h$ is \emph{smaller} than $D_k$ and 
$D_k$ is \emph{larger} than $D_h$.

\smallskip
\textit{Step~3.} For $k\ge 1$ we define subdendrites $E_k\supsetneq D_k$ of $D$
  such that $E_1=D$ and 
  each $E_k$ ($k\ge 2$) will contain one bush $D_{\ell_k}$ larger than $D_k$ 
  and infinitely many bushes smaller than $D_k$.
\smallskip

For every $k\ge 2$ find an integer $1\le \ell_k<k$ such that $d(x_k,x_{\ell_k})\to 0$
and $\ell_k\to\infty$ as $k\to\infty$; 
this is possible since the set $\{x_k\colon k\in\NNN\}$ is dense in $A$.
Put 
$$
  N_k = 
  \{k,\ell_k\}
  \ \sqcup\ 
  \{h>k\colon x_h\in(x_k,x_{\ell_k})\}.
$$
Define subdendrites $E_k$ of $D$ by
$$
  E_1=D,
  \qquad
  E_k = [x_k,x_{\ell_k}] 
  \cup
  \bigsqcup_{h\in N_k} D_h
  \quad(k\ge 2).
$$
Thus $E_k$ is the union of the arc $[x_k,x_{\ell_k}]$, bushes $D_k$ and $D_{\ell_k}$,
and all those bushes rooted at the points strictly between $x_k$ and $x_{\ell_k}$ which are smaller than the bush $D_k$.

For $k\ge 2$ we have $\HHh^1_d(E_k\cap A)=d(x_k,x_{\ell_k})$ and
$\HHh^1_d(E_k\setminus A)=\sum_{h\in N_k} \lambda_h \le \lambda_{\ell_k}+\sum_{i\ge k} \lambda_i = \lambda_{\ell_k}+q^k$.
Moreover, $\lambda_{\ell_k}\to 0$ because $\ell_k\to\infty$.
Therefore $\HHh^1_d(E_k) \to 0$, hence
\begin{equation}\label{EQ:Zk-length}
  \diam\nolimits_d(E_k) \to 0.
\end{equation}

\smallskip
\textit{Step~4.} We define surjective maps $\tilde\varphi_k\colon I\to E_k$ ($k\ge 1$).
\smallskip

Fix $k \ge 2$. Denote
$$
  t_h=\psi_0(x_h)
  \text{ for every } h\in N_k,
  \quad\text{and}\quad
  J_k=[t_k,t_{\ell_k}]\subseteq I;
$$
we will assume that $t_k<t_{\ell_k}$ (the case $t_k>t_{\ell_k}$ is analogous).
Note that the set $\{t_h\colon h\in N_k\}$ is dense in $J_k$.
We are going to define an auxiliary surjection $g_k$ from a compact real interval $J_k^+$ (containing $J_k$)
to $E_k$. The construction starts, roughly speaking, by taking the interval $J_k$ and  
``blowing-up'' every point $t_h$ ($h\in N_k$) into a closed interval $T_h$
of length $\lambda_h$. To be more precise, for every $h\in N_k$ put 
\begin{equation*}
	\Lambda_h = \sum\limits_{i\in N_k, \ t_{i}<t_h}  \lambda_i
\end{equation*}
(here $\Lambda_k=0$);
note that every $\Lambda_h$ is finite. 
Put $J_k^+=[t_k, t_{\ell_k}+\Lambda_{\ell_k}+\lambda_{\ell_k}]$
and let $\mu_k\colon J_k^+\to J_k$ be the unique nondecreasing (continuous) surjection
such that $\mu_k^{-1}(t_h)=T_{h}=[t_h+\Lambda_h, t_h+\Lambda_h+\lambda_h]$
for every $h\in N_k$.

We are ready to define a continuous surjective 
map $g_k\colon J_k^+\to E_k$ as follows.
\begin{itemize}
\item If $s\in T_{h}$ 
	for some $h\in N_k$, then put
	$g_k(s)=\varphi_h(s')$, 
	where the positions of $s'$ in $I$ is the same as that of $s$ in $T_h$;
	thus
	 $s'=(s-t_h-\Lambda_h)/\lambda_h\in I$.
	Note that $g_k(s)\in D_h\subseteq E_k$.
\item If $s\in J_k^+\setminus \bigcup_{h\in N_k} T_{h}$, then put
    $g_k(s)=\varphi_0(\mu_k(s))$; 
    here we have that $g_k(s)\in [x_k,x_{\ell_k}]\setminus\{x_h\colon h\in N_k\} \subseteq [x_k,x_{\ell_k}]=E_k\cap A$.
\end{itemize}
The map $g_k\colon J_k^+\to E_k$ has the following properties:
\begin{enumerate}
\item\label{EQ:gk-Th}
	$g_k(T_h)=D_h$ for every $h\in N_k$, the endpoints of $T_h$ being mapped to $x_h$ by \eqref{EQ:phipsi0};
	in particular, $g_k(t_k)=x_k$;
\item\label{EQ:gk-in-Dh}
	for $s\in J_k^+$ and $h\in N_k$, $g_k(s)\in D_h$ if and only if $s\in T_h$;
\item\label{EQ:gk-Ch}
	if $h\in N_k$ and $L$ is a nondegenerate closed interval in $T_h$, then $g_k(L)\in \CCc_h$
	(indeed, $g_k|_{T_h}\colon T_h\to D_h$ is the composition of a nonconstant linear map $T_h\to I$
	followed by the map $\varphi_h\colon I\to D_h$ and so $g_k(L)=\varphi_h(L')$, where $L'$ is a nondegenerate closed interval in $I$, i.e., $L'\in \CCc_I$, whence $g_k(L)\in \varphi_h(\CCc_I)=\CCc_h$);
\item\label{EQ:gk-Jk}
    $g_k(J_k^+\setminus \bigcup_{h\in N_k} T_{h})=[x_k,x_{\ell_k}] \setminus\{x_h\colon h\in N_k\}$;
\item\label{EQ:gk-surj}
 	$g_k$ is surjective by \eqref{EQ:gk-Th} and \eqref{EQ:gk-Jk};
\item\label{EQ:gk-cont}
	$g_k$  is continuous (to see this, use \eqref{EQ:gk-Th}, \eqref{EQ:gk-Jk} and the facts that 
	the maps $\mu_k$, $\varphi_0$ and $\varphi_h$ ($h\in N_k$) are continuous, and
	the families $\{T_h\colon h\in N_k\}$ and $\{D_h\colon h\in N_k\}$ are null);
\item\label{EQ:gk-lexp}
	for every closed nondegenerate subinterval $K$ of $J_k^+$,
	\begin{itemize}
	\item 
		if $g_k(K)\cap A$ is nondegenerate, then $g_k(K)\supseteq D_h$ for some (even for infinitely many) 
		$h\in N_k$;
	\item
		if $g_k(K)\cap A$ is degenerate or empty, that is, $K$ is a subinterval of some
		$T_{h}$, then $g_k(K)= D_h$ or $\HHh^1_d(g_k(K))\ge \varrho\abs{K}$.
	    To see this, let $K'\subseteq I$ be the interval whose position in $I$ is the same as the position of
	    $K$ in $T_h$;
	    i.e., ${K'}=(K-t_h-\Lambda_h)/\lambda_h$. Then $\abs{K'} = \abs{K}/\lambda_h$
	    and $g_k(K)=\varphi_h(K')$. Now, by 
        $\varrho$-length expansiveness of the map $\varphi_h\colon I\to D_h$ and the definition of $d$,
        $g_k(K)=\varphi_h(K')=D_h$ or
        \[
        	\frac{1}{\lambda_h} \HHh^1_d(g_k(K)) 
        	= \HHh^1_{d_k}(\varphi_h(K')) 
        	\ge \varrho \abs{K'}
        	= \frac{\varrho}{\lambda_h} \abs{K}.
        \]
	\end{itemize}
\end{enumerate}

Still for $k\ge 2$ we define a map $\tilde\varphi_k\colon I\to E_k$ via the composition
$\tilde\varphi_k=g_k\circ\nu_k$, where $\nu_k\colon I\to J_k^+$ is a 
continuous map such that $\nu_k(0)=t_k$ and 
for every closed nondegenerate subinterval $J$ of $I$, 
$\abs{\nu_k(J)}\ge \abs{J}$.\footnote{For instance, $\nu_k$ can 
be a piecewise linear map with constant slope and sufficiently many laps, each of which is mapped onto $J_k^+$.}
So, by \eqref{EQ:gk-lexp},
\begin{equation}\label{EQ:LEL-of-tilde-phi}
	\tilde\varphi_k(J) \text{ contains some } D_h,
	\qquad\text{or}\qquad
	\HHh^1_d(\tilde\varphi_k(J))\ge \varrho\abs{J}
	\quad\text{with}\quad
	\card(\tilde\varphi_k(J)\cap A)\le 1.
\end{equation}
By \eqref{EQ:gk-surj} and \eqref{EQ:gk-cont}, 
$\tilde\varphi_k$ is a continuous surjection and, since $g_k(t_k)=x_k$,
\begin{equation}\label{EQ:phi0}
\tilde\varphi_k(0)=x_k.
\end{equation}

So we have defined $\tilde\varphi_k$ for every $k\ge 2$.
For $k=1$, the map $\tilde\varphi_1\colon I\to E_1$ is defined analogously,
as $\tilde\varphi_1=g_1\circ \nu_1$.
The only differences are:
\begin{itemize}
\item in the definition of $g_1\colon J_1^+\to E_1$ we put $J_1=I$ and $J_1^+=[0,1+\Lambda]$, 
	where $\Lambda=\sum_{h\in\NNN}\lambda_h=q$;
	now the analogue of the property $g_k(t_k)=x_k$ is $g_1(0)=x_1$;
\item to get $\tilde\varphi_1(0)=x_1$ we choose $\nu_1\colon I\to J_1^+$ with $\nu_1(0)=0$.
\end{itemize}
Then \eqref{EQ:LEL-of-tilde-phi} and \eqref{EQ:phi0} work also for $k=1$.

\smallskip
\textit{Step~5.} We define a continuous map $f\colon D\to D$ such that $f(D_1)=D$ and 
$f(D_k)\supseteq D_{\ell_k}$ for every $k\ge 2$.
\smallskip

Define $f$ by
\begin{equation*}
	f(x) = 
	\begin{cases}
		x 
		&\text{if } x\in A;
	\\
		\tilde\varphi_k\circ \psi_k(x)  
		&\text{if } x\in D_k \text{ for some }k\in\NNN.
	\end{cases}
\end{equation*}
Since $\tilde\varphi_k\circ \psi_k(x_k)=\tilde\varphi_k(0)=x_k$
for every $k\ge 1$ by \eqref{EQ:phipsi0} and \eqref{EQ:phi0}, the map $f$ is well-defined.
Further, $f(D_k)=E_k$ for every $k\in\NNN$; thus
\begin{equation}\label{EQ:fDk}
	f(D_1)=D 
	\qquad\text{and}\qquad
	f(D_k)\supseteq D_{\ell_k} 
	\quad\text{for every }k\ge 2.
\end{equation}
The continuity of $f$ is obvious since it is continuous on $A$ and on each $D_k$,
and the $f$-images of the bushes $D_k$ form a null family by \eqref{EQ:Zk-length}.

\smallskip
\textit{Step~6.} We claim that, for every $k\in\NNN$ and every $C\in \CCc_k$
(see \eqref{EQ:Ck}),
at least one of the following holds:
\renewcommand{\theenumi}{\alph{enumi}}
\begin{enumerate}
	\item\label{ENUM:exact-1} $f(C)$ contains some $D_h$; 
	\item\label{ENUM:exact-2} $\HHh^1_d(f(C)) \ge \varrho^2 \cdot \HHh^1_d(C)$
	and $f(C)\in\CCc_h$ for some $h\in\NNN$.
\end{enumerate}
\renewcommand{\theenumi}{enumi}
\smallskip

To see this, we use \eqref{EQ:LEL-of-tilde-phi}
and $\varrho$-length expansiveness of $\psi_k$. Indeed, 
fix any $k\in\NNN$ and $C\in\CCc_k$, and denote
the (nondegenerate) closed interval $\psi_k(C)$ by $J$.
Since $C\subseteq D_k$, we have
\[
	f(C)
	=\tilde\varphi_k(\psi_k(C))
	= \tilde\varphi_k(J)
	= g_k(\nu_k(J)).	
\]
Assume that $f(C)$ does not satisfy \eqref{ENUM:exact-1},
i.e., it contains no $D_h$.
Then $J\ne I$ (since otherwise $f(C)=\tilde\varphi_k(I)=E_k\supseteq D_k$)
and so, by $\varrho$-length expansiveness of $\psi_k$,
$J=\psi_k(C)\in C_I$ is nondegenerate and
$\abs{J}\ge\varrho\cdot \HHh^1_d(C)$.
Further, 
since we assume that $\tilde\varphi_k(J)$ contains no $D_h$, \eqref{EQ:LEL-of-tilde-phi} gives that
$\HHh^1_d(\tilde\varphi_k(J))\ge \varrho\abs{J}$,
hence $\HHh^1_d(f(C)) \ge \varrho^2 \cdot \HHh^1_d(C)$,
and $\card(\tilde\varphi_k(J)\cap A)\le 1$. 
This cardinality condition gives that $\tilde\varphi_k(J)\subseteq D_h$ for some $h\in N_k$.
Since $\tilde\varphi_k(J)= g_k(\nu_k(J))$, the property \eqref{EQ:gk-in-Dh} of $g_k$
shows that the (nondegenerate) closed interval $\nu_k(J)$ is a subset of $T_h$. Then,
by the property \eqref{EQ:gk-Ch} of $g_k$, we get that
$f(C)=g_k(\nu_k(J))\in \CCc_h$.
So we have \eqref{ENUM:exact-2}.

\smallskip
\textit{Step~7.} We prove that for every nonempty open set $U$ in $D$ there is $h\ge 1$ and $n\in\NNN_0$
such that 
\[
  f^n(U)\supseteq D_h.
\]

Since $A$ is nowhere dense, we may assume
that $U$ is a subset of $D_k$ for some $k\ge 1$. 
By \eqref{EQ:Ck} we may fix some $C\in\CCc_k$ lying in $U$.
Since $\varrho>1$ and $D$ has finite length, the iterative use of Step~6
yields that $f^n(C)\supseteq D_{h}$ for some $n$ and $h$.

\smallskip
\textit{Step~8.} We finish the proof by showing that $f$ is exact.
\smallskip

In view of Step~7 it is sufficient to prove that for every $h_0\ge 1$ there is $m$
such that 
\[
  f^m(D_{h_0})=D.
\]
If $h_0=1$ then $f(D_1)=D$ by \eqref{EQ:fDk}, and we are done with $m=1$.
Now let $h_0\ge 2$ and put $h_1=\ell_{h_0}$ (see the definition of integers $\ell_k$ in the beginning of Step~3).
By \eqref{EQ:fDk}, $f(D_{h_0})\supseteq D_{h_1}$ and $1\le h_1<h_0$. 
If $h_1=1$ then we are done with $m=2$. If $h_1>1$, we continue 
by taking $h_2=\ell_{h_1}$ and so on. Since $h_0>h_1>h_2>\dots\ge 1$,
in a finite number $n$ of steps we obtain that $h_n=1$.
Then, $f^{n+1}(D_{h_0})\supseteq f(D_{1})=D$ and we are done with $m=n+1$.
This finishes the proof.
\end{proof}

\section{Proof of Theorem~\ref{T:dendrite_gch=gech}: (\ref{ENUM:T1}) implies (\ref{ENUM:T2})}\label{S:thm12}

The following two lemmas generalize ideas from \cite[p.~49]{Mur00}.

\begin{lemma}\label{L:union-k}
	Let $(X,d)$ be a compact metric space and $f\colon X\to X$ be 
	a continuous map. Assume that, for some $k\in \NNN$, there are
	$f$-invariant closed subsets $Z_1,\dots,Z_k$ of $X$ such that 
	\begin{enumerate}
		\item\label{ENUM:union-1} $X=Z_1\cup\dots\cup Z_k$;
		\item\label{ENUM:union-2} $Z_i\cap Z_j\ne\emptyset$ for every $i,j$;
		\item\label{ENUM:union-3} $f|_{Z_i}\colon Z_i\to Z_i$ is strongly mixing for every $i$.
	\end{enumerate}
	Then $f$ is generically $\eps$-chaotic for any $0<\eps<(1/2)\min_i\diam Z_i$.
\end{lemma}

\begin{proof}
	We may assume that $X$ is nondegenerate, otherwise the claim is trivial.
If $k=1$, the lemma follows from the fact, mentioned already in 
Section~\ref{S:intro}, that a weakly mixing map on a (nondegenerate) compact metric space 
$X$ is generically $\eps$-chaotic for every $0<\eps<\diam X$. From now on  
assume that $k\geq 2$. 

If $B$ is an open ball in $X$, we can write $B = (B\cap Z_1) \cup \dots \cup 
(B\cap Z_k)$. Since the sets $B\cap Z_i$ are closed in 
$B$, there is $i$ such that $B\cap Z_i$ has nonempty interior in $B$. In other 
words, 
\begin{equation}\label{EQ:subball}
\text{every ball contains a ball lying in one of the sets $Z_i$.}
\end{equation}

To prove the lemma, it is clearly sufficient to show that $f$ is generically 
$\eps$-chaotic for any $\eps< \eta /2$ whenever $0< \eta < \min_i\diam Z_i$. 
Therefore fix such an $\eta$. By Proposition~\ref{P:B1-B2}, it is 
sufficient to prove that \ref{ENUM:B1} and \eqref{EQ:Sens-eta} in \ref{ENUM:B2} are satisfied by the 
family of all open balls in $X$.

To prove \ref{ENUM:B1}, let $B_1,B_2\subseteq X$ be open balls. To show that  
$\liminf_{n\to\infty} d(f^n(B_1),f^n(B_2))=0$, we may assume, in view 
of~\eqref{EQ:subball}, that $B_1\subseteq Z_i$ and $B_2\subseteq Z_j$ for some 
$i$ and $j$. Let $\delta>0$. By \eqref{ENUM:union-2} there is $x_0\in Z_i\cap 
Z_j$. Since the restrictions of $f$ to $Z_i$ and $Z_j$ are strongly mixing,
there is $n_0$ such that, for every $n\ge n_0$, both $d(x_0,f^n(B_1))$ and 
$d(x_0,f^n(B_2))$ are smaller than $\delta/2$. Hence $d(f^n(B_1),f^n(B_2)) < 
\delta$ for every $n\ge n_0$. We have proved that
$\lim_{n\to\infty} d(f^n(B_1),f^n(B_2))=0$.

To prove \eqref{EQ:Sens-eta} in \ref{ENUM:B2}, choose any open ball $B\subseteq X$. To show that 
$\limsup_{n\rightarrow \infty}\diam f^n(B)>\eta$, again by~\eqref{EQ:subball} we may
assume that $B\subseteq Z_i$ for some $i$. Put $g=f|_{Z_i}$.
Since $\eta<\diam Z_i$, there are nonempty open
sets $U,V\subseteq Z_i$ such that $d(U,V)>\eta$. 
By transitivity of $g\times g$, there is an increasing sequence 
$\{n_k\}_k$ of positive integers such that every $(g\times g)^{n_k}(B\times B)$ 
intersects $U\times V$. Hence $\diam g^{n_k}(B)\ge d(U,V)>\eta$ for every $k$,
which proves that $\limsup_{n\rightarrow \infty}\diam f^n(B)>\eta$.
\end{proof}

\begin{lemma}\label{L:union-inf}
	Let $X$ be a compact metric space and $f\colon X\to X$ be a continuous map.
	Assume that there is an increasing sequence of $f$-invariant closed sets 
	$X_i$ ($i\in\NNN$) such that $X=\bigcup_{i=1}^\infty X_i$
	and $f|_{X_i}\colon X_i\to X_i$ is generically chaotic for every $i$. 
	Then $f$ is generically chaotic.
\end{lemma}
\begin{proof}
	Denote by $\nLY(f)$ the set of pairs $(x,y)\in X^2$ which are not Li-Yorke 
	for $f$. Since every $f|_{X_i}$ is generically chaotic, the sets $X_i^2\cap 
	\nLY(f)$ are of the first category in $X_i^2$, hence of the first category in $X^2$.
	Since $X^2=\bigcup_{i=1}^\infty X_i^2$ due to the fact that $X_i\subseteq X_{i+1}$ 
	for every $i$, the set $\nLY(f)$ is of the first category in $X^2$. This shows 
	that $f$ is generically chaotic.
\end{proof}
One can see that a slightly stronger lemma is true. It is sufficient to assume 
that $X=\bigcup_{i=1}^\infty  X_i\cup Y$, where $Y$ is of the first category in $X$.

By combining the previous two lemmas we get the following proposition.

\begin{proposition}\label{P:GCH-not-eps}
	Let $X$ be a compact metric space and $f\colon X\to X$ be a continuous map.
	Assume that there are $f$-invariant closed sets $X_i$ ($i\in\NNN$) such that 
	\begin{enumerate}
		\item $X=\bigcup_{i=1}^\infty  X_i$;
		\item\label{ENUM:Xi-interior}
			every $X_i$ has nonempty interior;
		\item\label{ENUM:diamXi}
			$\diam X_i\to 0$;
		\item $X_i\cap X_j\ne\emptyset$ for every $i\ne j$;
		\item  $f|_{X_i}\colon X_i\to X_i$ is  strongly mixing for every $i$.
	\end{enumerate}
	Then $f$ is generically chaotic but not generically $\eps$-chaotic for any $\eps>0$.
\end{proposition}
\begin{proof}
	Denote $Y_k=\bigcup_{i=1}^k X_i$, $k\in\NNN$. By Lemma~\ref{L:union-k}, $f|_{Y_k}\colon Y_k\to Y_k$
	is generically chaotic for every $k$. Since $X=\bigcup_{k=1}^\infty  Y_k$, $f$ is generically chaotic by Lemma~\ref{L:union-inf}. By \eqref{ENUM:Xi-interior} and \eqref{ENUM:diamXi}, there are arbitrarily small invariant sets 
	with nonempty interiors. This implies that \ref{ENUM:B2} is not satisfied by the family of open balls in $X$.
	Therefore, by Proposition~\ref{P:B1-B2}, $f$ is not generically $\eps$-chaotic for any $\eps>0$.
\end{proof}

\begin{lemma}\label{L:gench-not-eps}
	Let $D$ be a nondegenerate dendrite and $A\subseteq D$ be either a nowhere dense nondegenerate 
	subdendrite, or a singleton $\{a\}$ such that $a$ is of infinite order in $D$. 
	Then there is a generically chaotic map 
	$f\colon D\to D$ which is not generically $\eps$-chaotic for any $\eps>0$.
\end{lemma}
\begin{proof}
	We may assume that $A$ is either a nowhere dense arc or $A=\{a\}$. 
	By the assumptions, in either case the set $D\setminus A$ has infinitely many components $C_i$ ($i\in\NNN$).
	By Lemma~\ref{L:doplnkySubdendr-Hranica}, $\Bd(C_i)=\{c_i\}\subseteq A$ and, since $A$ is nowhere dense, the set $\{c_i\colon i\in\NNN\}$
	is dense in $A$. Further, the open connected sets $C_i$ form a null family.
	
	Assume first that $A=\{a\}$. By Proposition~\ref{P:exact}, every dendrite $\overline{C_i}$ admits
	an exact map $f_i$ fixing the point $a$.
	The map $f\colon D\to D$ such that $f(x)=f_i(x)$ if $x\in \overline{C_i}$ is well defined and continuous.
	By Proposition~\ref{P:GCH-not-eps}, $f$ is generically chaotic but not generically $\eps$-chaotic for any $\eps>0$. 
	
	Now let $A$ be a nowhere dense arc. 	
	Fix $a\in A$; for simplicity we can choose $a\notin \{c_i\colon i\in\NNN\}$. 
	Choose a nested sequence of subarcs $A_j$ of $A$ such that
	$A_1=A$ and $\bigcap_{j=1}^\infty A_j = \{a\}$.
	By induction one can obviously construct a partition $\NNN=\bigsqcup_{j=1}^\infty N_j$
	with infinite sets $N_j$
	such that the (nondegenerate) dendrites $E_j$ ($j\in\NNN$) defined by 
	\[
	E_j = A_j \sqcup  \bigsqcup_{i\in N_j} C_i
	\]  
	are such that
	\begin{enumerate}[label=(\alph*)]
		\item\label{ENUM:not-Aj} $A_j=E_j\cap A$ is nowhere dense in $E_j$ for every $j$.
	\end{enumerate}	
	We clearly have the following:
	\begin{enumerate}[resume*]
		\item\label{ENUM:not-union} $D=\bigcup_{j=1}^\infty E_j$;
		\item\label{ENUM:not-y0} $a\in E_j$ for every $j$;
		\item\label{ENUM:not-ZjZk} $E_j\cap E_k=A_{k} \subseteq A$ for every $j< k$;
		\item\label{ENUM:not-diamZj} $\diam E_j\to 0$ as $j\to\infty$ (because $a\notin\{c_i\colon i\in\NNN\}$);
		\item\label{ENUM:not-intEj} $E_j$ has nonempty interior in $D$ for every $j$
		(in fact, every $E_j$ contains some $C_i$ and $C_i$ is nonempty and open in $D$);
		\item\label{ENUM:not-ZjA} there is a sequence of positive reals $\delta_j\to 0$ 
		such that $E_j\subseteq B_{\delta_j}(a)$ for every $j$,
		i.e., in the Hausdorff metric the sequence $E_j$ converges to the singleton $\{a\}$
		(this follows from \ref{ENUM:not-y0} and \ref{ENUM:not-diamZj}).
	\end{enumerate}

	Using \ref{ENUM:not-Aj} and Proposition~\ref{P:exact}, for every $j$ there is an exact map
	\begin{equation}\label{EQ:L:gench-not-eps:1}
	    f_j\colon E_j\to E_j
  		\quad\text{such that}\quad
		f_j(x) = x 
		\quad\text{for every }x\in A_j. 
	\end{equation}
	Define $f\colon D\to D$ by $f(x)=f_j(x)$ provided $x\in E_j$ ($j\in\NNN$).
	This map is well-defined by \ref{ENUM:not-union}, 
	\ref{ENUM:not-ZjZk} and \eqref{EQ:L:gench-not-eps:1}.
	We prove that $f$ is continuous. If $x\ne a$ then, by \ref{ENUM:not-ZjA}, there is an open neighbourhood of $x$ covered by finitely many sets $E_j$; thus $f$ is continuous at $x$ by the pasting lemma.
	To prove continuity at the point $a$, fix a neighbourhood $U$ of $a$. By \ref{ENUM:not-ZjA}
	there exists $N$ such that $\bigcup_{j>N} E_j\subseteq U$. Since $f_1,\dots,f_N$ are continuous at $a$ and $a$
	is their common fixed point,
	there is a neighbourhood $V\subseteq U$ of $a$ such that $f_j(V)\subseteq U$ for every $j\le N$.
	Then $f_j(V)\subseteq U$ for all $j$ and so $f(V)\subseteq U$.

	By Proposition~\ref{P:GCH-not-eps}, $f$ is generically chaotic but not generically $\eps$-chaotic for any $\eps>0$. 
\end{proof}

\begin{proof}[Proof of Theorem~\ref{T:dendrite_gch=gech}: $(\ref{ENUM:T1}) \Rightarrow (\ref{ENUM:T2})$]
	Assume that (\ref{ENUM:T2}) is not true, i.e.,  $D$ contains a nowhere dense nondegenerate subdendrite
	or $D$ contains a point of infinite order. In either case, Lemma~\ref{L:gench-not-eps} shows
	that $D$ admits a generically chaotic selfmap which is not generically $\eps$-chaotic for any $\eps>0$.
	Thus (\ref{ENUM:T1}) from Theorem~\ref{T:dendrite_gch=gech} is not satisfied.
\end{proof}

We add the following simple observation used at the end of Section~\ref{S:intro}.

\begin{proposition}\label{P:Cantor}
	A Cantor set admits a generically chaotic map which is not generically $\eps$-chaotic for any $\eps>0$.
\end{proposition}
\begin{proof}
	Let $D$ be an $\omega$-star with branch point $z$. Let $A_i$ ($i\in\NNN$) be the closures of the components
	of $D\setminus\{z\}$. 
	In each $A_i$ choose a Cantor set $Z_i$ containing $z$.
	Then $X=\bigcup_{i=1}^\infty Z_i$ is a Cantor set.
	Every $Z_i$ admits a strongly mixing map $f_i\colon Z_i\to Z_i$ with $f_i(z)=z$.
	Then $f\colon X\to X$ defined by $f(x)=f_i(x)$ for every $i\in \NNN$ and $x\in Z_i$
	is continuous. By Proposition~\ref{P:GCH-not-eps}, $f$ is generically chaotic but not generically $\eps$-chaotic for any $\eps>0$. 
\end{proof}

Finally, here we present the example, suggested by an anonymous referee and mentioned in Section~\ref{S:intro}, of a generically chaotic map which is not sensitive.
	
\begin{example}[cf. Corollary~4.2 in \cite{HY02}]\label{Ex:gch-not-sensitive}
Take any  topologically transitive, nonminimal homeomorphism $h$
on a compact metric space $X$
which is almost equicontinuous
(for a construction of such a homeomorphism see e.g.~\cite[Theorem~4.2]{AAB}).
Then the transitive points are equicontinuity
points \cite[Theorem~2.4]{AAB}, 
$h$ is uniformly rigid \cite[Corollary~3.7]{AAB}
and the union $M$ of all minimal sets is not dense \cite[Theorem~2.5]{AAB}. 
The closure $\overline M$ is an invariant nowhere dense set
(if $\overline M$ has nonempty interior, then it contains a transitive point 
and so $\overline M=X$, contradicting the fact that $M$ is not dense).
By collapsing $\overline M$ to a point,
we obtain a system $(\tilde X, \tilde h)$ where $\tilde h$ is a homeomorphism of a compact metric space $\tilde X$.
This new system has the following properties. 
First, since all the equicontinuity points (i.e., the transitive points)
of $h$ are outside the invariant 
nowhere dense closed set $\overline M$,
it is straightforward to show that also $\tilde h$ is almost equicontinuous
(alternatively, one can use \cite[Lemma~1.6]{GW}).
Further, $\tilde h$ has a fixed point which is
the unique minimal set,
and so the new system is proximal 
\cite[Proposition~2.2]{AK}. Finally, 
$\tilde h$ is uniformly rigid \cite[Corollary~3.7]{AAB}
and so every pair of points is recurrent.
That is, the whole
space $\tilde X$ is strongly scrambled (meaning that every pair of points is proximal 
and recurrent) 
and so the system $(\tilde X, \tilde h)$ is obviously generically chaotic.
Being almost equicontinuous, it is not sensitive.
\end{example}

\setcounter{section}{0}
\renewcommand{\thesection}{\Alph{section}}
\section{Appendix: Li-Yorke chaotic dendrite map which is not Li-Yorke 
	$\eps$-chaotic}\label{S:appendix}

The purpose of this appendix is to show that an analogue of Theorem~\ref{T:dendrite_gch=gech}
does not hold with generic chaos replaced by Li-Yorke chaos. 
We start with a construction of a system which resembles that of Floyd-Auslander \cite{Fl49,Aus88}, see also \cite{HJ97}. However, our pattern for producing a family of subrectangles from a given rectangle is very different. Therefore the obtained homeomorphism will not be minimal (even not transitive) and in fact will have an appropriate invariant Cantor set.

Consider the alphabet $\AAa=\{0,1,2\}$, the set $\AAa^*=\bigcup_{n\in \NNN_0} \AAa^n$ of finite words (here $\AAa^0$ contains just the empty word $\emptyset$) and the set $\Sigma=\AAa^{\NNN_0}$ of infinite words over the alphabet $\AAa$. The elements $\alpha \in \Sigma$ are infinite sequences $\alpha = \alpha_0 \alpha_1 \dots$, $\alpha_i$ being the $i$-th coordinate of $\alpha$. Analogously, the finite words are written in the form of finite sequences; in the usual way we can concatenate them. The set $\Sigma$ together with addition $\bmod 3$ with carry from the left to the right is the $3$-adic group. Below we abbreviate $\alpha+10^\infty$ to $\alpha+1$;
then $\alpha+n$ has the usual meaning for every $n\in\ZZZ$.

For a rectangle $K=[a,b]\times [c,d]$ let $K_i$ ($i\in\AAa$) be the subrectangles of 
$K$ given by
\[
	K_i = \left[a+\frac{2i(b-a)}5,a+\frac{(2i+1)(b-a)}5\right]
	\times \left[c,c+\theta_i(d-c)  \right],
\]
where $\theta_i = 1$ if $i=1$ and $\theta_i=1/3$ otherwise. Starting with $K_\emptyset=[0,1]^2$
and applying this pattern inductively, we obtain rectangles $K_a$ for every $a\in\AAa^*$;
for every $a\in\AAa^*$ and $i\in\AAa$ we put $K_{ai} = (K_a)_i$ (recall that $\emptyset i =i$). 
For every $n\in\NNN_0$ define $X_n=\bigcup_{a\in\AAa^n} K_a$, see Figures~\ref{Fig:FA-1} and \ref{Fig:FA-2}, and $X = \bigcap_{n\in\NNN_0} X_n$. One can see that 
\[
	X 
	= \bigcup_{\alpha\in\Sigma} K_\alpha,
	\qquad\text{where}\quad
	K_\alpha 
	= \{x_\alpha\} \times J_\alpha,
\]
with $x_\alpha=\sum_{i=0}^\infty (2\alpha_i)/5^{i+1}$ and 
\[
	J_\alpha=[0,3^{-\ell_\alpha}],
	\qquad
	\ell_\alpha = \card\{i\in\NNN_0\colon \alpha_i\in\{0,2\} \}
\]
(we adopt the convention $3^{-\aleph_0}=0$).
The sets $K_\alpha$ will be called \emph{(vertical) fibres} of $X$; $K_{\alpha}$ is the fibre above $x_\alpha$. 

Define a map $H\colon X\to X$ by 
\[
	H(x_\alpha,y) = (x_{\alpha+1},h_\alpha(y))
	\qquad
	(\alpha\in \Sigma,\ y\in J_\alpha),
\]
where 
\begin{equation}\label{EQ:LY-galpha}
	h_\alpha\colon J_\alpha\to J_{\alpha+1},\qquad
	h_\alpha(y) = 3^{\ell_\alpha - \ell_{\alpha+1}} y,
\end{equation}
is an increasing linear surjection. Figures~\ref{Fig:FA-1} and \ref{Fig:FA-2} show the form of the fibre maps $h_{\alpha}$ except on the fibres lying in the rightmost rectangle.

\begin{figure}[ht!]
	\centering{\includegraphics{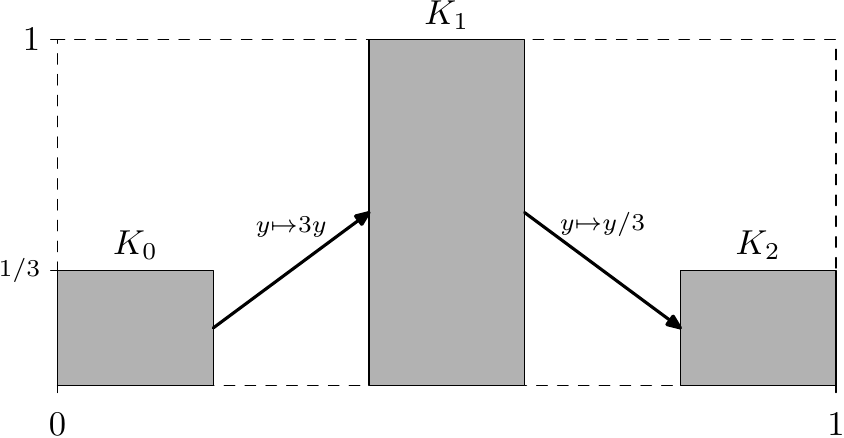}}
	\caption{The set $X_1 = K_0 \cup K_1 \cup K_2$}
	\label{Fig:FA-1}
\end{figure}

\begin{figure}[ht!]
	\centering{\includegraphics{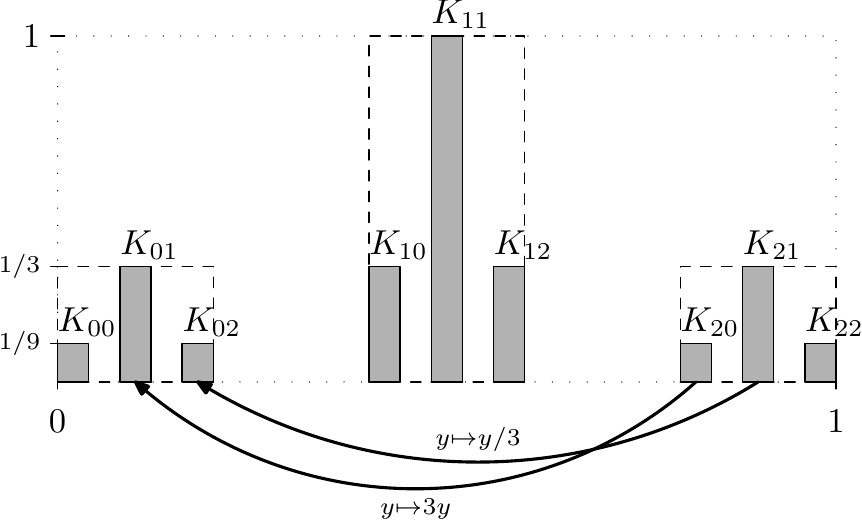}}
	\caption{The set $X_2 = K_{00} \cup K_{01} \cup \dots \cup K_{22}$}
	\label{Fig:FA-2}
\end{figure}

Basic properties of $X$ and $H$ are summarized in the following lemma.
We omit a proof, since it is straightforward and analogous to the proof of the corresponding properties of the Floyd-Auslander system (the property \eqref{ENUM:L:LY-H:scambled} can be proved similarly as~\cite[Proposition 2(b)]{CSS04}).

\begin{lemma}\label{L:LY-H}
Let $X$ and $H$ be defined as above. Then the following hold.
\begin{enumerate}
\item $X$ is a compact subset of the unit square. The projection of $X$ onto the first coordinate is a Cantor set $C_1=\{x_\alpha\colon \alpha\in\Sigma\}$ and the connected components 
of $X$ are the fibres $K_\alpha$ ($\alpha\in\Sigma$).
\item Every degenerate fibre is a singleton with second coodinate zero. The fibre $K_\alpha$ is nondegenerate if and only if $\alpha=1^\infty + m$ for some $m\in\ZZZ$ (i.e., if only finitely many coordinates of $\alpha$ are different from $1$). Thus the nondegenerate fibres are the fibres above the full orbit (under the homeomorphism $x_{\alpha} \mapsto x_{\alpha +1}$) of the point $x_{\alpha}$, $\alpha = 1^{\infty}$.
\item The map $H$ is a homeomorphism and $H^n(K_\alpha)=K_{\alpha+n}$
for every $\alpha\in\Sigma$ and $n\in\ZZZ$.
\item For every nondegenerate fibre $K_{\alpha}$ we have 
	\[
		\liminf_{n\to\infty} \diam H^n(K_\alpha) = 0,
		\qquad
		\limsup_{n\to\infty} \diam H^n(K_\alpha) = \frac 13.
	\]
\item\label{ENUM:L:LY-H:scambled} Maximal scrambled sets of $H$ are the nondegenerate vertical fibres $K_\alpha$.
\item For every $\eps>0$ there is $n_\eps>0$ such that every $\eps$-scrambled set of $H$ 
has cardinality at most $n_\eps$.
\end{enumerate}	
\end{lemma}

Now let $C_2$ denote the Cantor ternary set. Then, by \eqref{EQ:LY-galpha}, the
(closed) set
$Y=X \cap (C_1\times C_2)$ 
is strongly $H$-invariant, i.e., $H(Y)=Y$.
Obviously, $Y$ is a Cantor set and 
$K_\alpha\cap Y$ is a Cantor set for every nondegenerate fibre $K_\alpha$.
Then Lemma~\ref{L:LY-H} immediately yields the following corollary.

\begin{corollary}\label{L:LY-G}
	There is a Cantor homeomorphism 
	which is Li-Yorke chaotic but not Li-Yorke $\eps$-chaotic
	for any $\eps>0$.
\end{corollary}

Let $(Y,G)$ be the system from Corollary~\ref{L:LY-G} 
and let $D$ be the Gehman dendrite \cite[p.~42]{Ge25}.
(Recall that the \emph{Gehman dendrite} is a topologically unique 
dendrite $D$ whose all branch points are of order $3$ and 
the set of endpoints $\End(D)$ is a Cantor set.) 
We may assume that $\End(D)=Y$. By \cite[Theorem~B]{Ro20}, 
there is a continuous map $F\colon D\to D$ such that
the restriction of $F$ onto $\End(D)=Y$ is $G$, and every point of
$D\setminus\End(D)$ is eventually mapped to one fixed point
which also lies in $D\setminus\End(D)$. 
This clearly implies that the scrambled sets of $F$ coincide with the scrambled sets of $G$. 
Hence $(D,F)$ is Li-Yorke chaotic but not Li-Yorke $\eps$-chaotic
for any $\eps>0$.
Thus the Gehman dendrite is a (completely regular) dendrite (with all points of finite order)
on which Li-Yorke chaos and Li-Yorke $\eps$-chaos are not equivalent.
However, more can be said.

Since every dendrite with uncountably many endpoints contains a copy of the Gehman dendrite
\cite[Proposition~6.8]{AC01}, 
and dendrites are absolute retracts, we immediately get the following fact.

\begin{proposition}\label{P:LY-chaos}
	Every dendrite with uncountably many endpoints
	admits a Li-Yorke chaotic map which is not Li-Yorke $\eps$-chaotic
	for any $\eps>0$.
\end{proposition}

	\medskip
	\noindent\emph{Acknowledgements.}
	The authors are very obliged to the anonymous referees for a list 
	of suggested improvements of the paper.
This work was supported by the Slovak Research and Development Agency 
under contract No.~APVV-15-0439 and by VEGA grant 1/0158/20.


\end{large}
\end{document}